\documentclass{article}
\usepackage{amsmath}
\usepackage{amsfonts}
\usepackage{amsthm}
\usepackage{amssymb}
\usepackage{url}
\usepackage{graphicx} 
\usepackage{multirow} 
\newtheorem{theorem}{Theorem}[section]
\newtheorem{corollary}{Corollary}[section]
\newtheorem{remark}{Remark}[section]
\newtheorem{proposition}{Proposition}[section]
\newtheorem{lemma}{Lemma}[section]

\newtheorem{innercustomthm}{Assumptions}
\newenvironment{customthm}[1]
  {\renewcommand\theinnercustomthm{#1}\innercustomthm}
  {\endinnercustomthm}

\newtheorem{definition}{Definition}[section]
\newtheorem{assumption}{Assumptions}[section]
\usepackage{float}
\usepackage{subcaption}
\usepackage{algorithm}
\usepackage{algpseudocode}
\usepackage{xcolor} 

\newcommand{\comment}[1]{}

\title{Analysis of Randomized Householder-Cholesky QR Factorization with Multisketching}
\author{Andrew J. Higgins\thanks{Center for Computing Research, Sandia National Laboratories, Albuquerque, New Mexico, USA (\texttt{ajhiggi@sandia.gov}, \texttt{egboman@sandia.gov}, \texttt{iyamaza@sandia.gov}).} 
\and Daniel B. Szyld\thanks{Temple University, Philadelphia, Pennsylvania, USA (\texttt{szyld@temple.edu}).} 
\and  Erik G. Boman\footnotemark[1]
\and Ichitaro Yamazaki\footnotemark[1]}
\date{\today}
\begin{document}
\maketitle
\begin{abstract}
CholeskyQR2 and shifted CholeskyQR3 are two 
state-of-the-art
algorithms for computing tall-and-skinny QR factorizations since they attain high performance on current computer architectures. 
However, to guarantee stability,
for some applications, CholeskyQR2 faces a prohibitive restriction on the condition number of 
the underlying matrix to factorize. Shifted CholeskyQR3 is stable but has $50\%$ more
computational and communication costs than CholeskyQR2. 
In this paper,
a randomized QR algorithm called Randomized Householder-Cholesky (\texttt{rand\_cholQR}) is proposed and
analyzed. 
Using one or two random sketch matrices, it is  proved that 
with high probability, 
its orthogonality error is bounded by a constant of the order of 
unit roundoff 
for any numerically full-rank matrix,  
and hence it is as stable as shifted CholeskyQR3. 
An evaluation of the 
performance of \texttt{rand\_cholQR} on an NVIDIA A100 GPU 
demonstrates that for tall-and-skinny matrices, \texttt{rand\_cholQR} with multiple sketch matrices is nearly as fast as, or in some cases faster than, CholeskyQR2. Hence, compared to CholeskyQR2, \texttt{rand\_cholQR} is more stable with almost no extra computational or memory cost, and therefore a superior algorithm both in theory and practice.
\end{abstract}



\section{Introduction}
Computing the QR factorization of tall-and-skinny matrices is a critical component of many scientific and engineering applications, including the solution of least squares problems, block orthogonalization kernels for solving linear systems and eigenvalue problems within block or $s$-step Krylov methods, dimensionality reduction methods for data analysis like Principal Component Analysis, and many others. 
The current state-of-the-art QR algorithms for tall-and-skinny matrices are the CholeskyQR2 
and shifted CholeskyQR3 algorithms~\cite{sCholQR3, CholeskyQR2Def}, 
thanks to their communication-avoiding properties along with their 
exploitation of vendor provided highly-optimized dense linear algebra subroutines~\cite{lapackUserGuide, cuda, cuSOLVER}.
However, CholeskyQR2 may fail to accurately factorize a matrix $V$ when 
its condition number $\kappa(V) \gtrapprox \textbf{u}^{-1/2}$, where 
$\textbf{u}$ is unit roundoff \cite{CholeskyQR2ErrAnalysis}. 
Shifted CholeskyQR3 is numerically stable as long as $\kappa(V) \lessapprox \textbf{u}^{-1}$, but it requires over $50\%$ more computational and communication cost than CholeskyQR2 \cite{sCholQR3}. 
Although more stable communication-avoiding algorithms exist, such as TSQR \cite{tsqrPerf},
they rely on Householder QR factorizations of potentially large matrices, and are often significantly slower than CholeskyQR2 in practice 
\cite{CholeskyQR2Def}.

In this paper, we first present and analyze a randomized algorithm called \texttt{randQR} 
for orthogonalizing the columns of a tall-and-skinny matrix with respect to a specific 
inner product. 
In order to reduce the cost of the computations, we propose to use ``multisketching," i.e., the use
of two consecutive sketch matrices, within \texttt{randQR}. Using \texttt{randQR} with multisketching as a preconditioner for less stable QR factorizations can be an efficient strategy for computing the true QR factorization of a matrix, which leads to the primary focus of this paper, which is an algorithm called \texttt{rand\_cholQR} for computing the true QR 
factorization of a tall-and-skinny matrix $V$.  
Our approach is general in the sense that our analysis
applies to any two $\epsilon$-subspace embedding sketching matrices (see Section~\ref{sec:randSketchTheory}
for definitions), but is specifically motivated by the use of a large sparse sketch followed by a smaller dense sketch, such as
a Gaussian or Radamacher sketch \cite{GaussSize}.
Our analysis applies 
in particular to Count-Gauss (one application of CountSketch followed by a Gaussian
sketch), as described in
\cite{Kapralov_Potluru_PMLR_16,Sobczyk_Gallopoulos_SIMAX_21, Sobczyk_Gallopoulos_TOMS_22}.

We prove that with high probability, 
the orthogonality error of \texttt{rand\_cholQR} is on the order of unit roundoff for any numerically 
full-rank matrix $V$ (i.e., $\kappa(V) \lessapprox \textbf{u}^{-1}$)
and hence it is as stable as shifted CholeskyQR3 and it is significantly more numerically stable than CholeskyQR2. 
Our numerical experiments ilustrate the theoretical results.
In addition, the \texttt{rand\_cholQR} algorithm may be implemented using the same basic linear algebra kernels as CholeskyQR2. Therefore, it is simple to implement and has the same communication-avoiding properties. We perform a 
computational study on a state-of-the-art GPU to demonstrate that \texttt{rand\_cholQR} can perform up to $4\%$ faster than CholeskyQR2 and $56.6\%$ faster than shifted CholeskyQR3, while significantly improving the robustness of CholeskyQR2.

In summary, our primary contribution consists of a new error analysis of a multisketched randomized QR algorithm, proving it can be safely used for matrices of larger condition number than CholeskyQR2 can handle.
This analysis also applies to the case of one sketch, improving upon the
existing results.
Our implementation confirms and illustrates the theory developed in this paper. Our secondary contribution is a computational study that tangibly demonstrates that the multisketched algorithm has superior performance over the existing single sketch algorithm and similar performance to the high-performance but less stable CholeskyQR2 algorithm.

%

In Section \ref{sec:relWork}, we begin by discussing prior work on similar topics. Then, in Section \ref{sec:randSketchTheory}, we present some preliminary definitions and known results from randomized linear algebra relevant to this work.
We follow with Section~\ref{multi:sec}, where we present multisketching on a conceptual level, and how to incorporate it into a randomized QR factorization (\texttt{rand\_cholQR}). We also discuss performance considerations for \texttt{rand\_cholQR} compared to other high performance tall-skinny QR algorithms, leading to the motivation as to why
multisketching is recommended.
In Section~\ref{sec:errAnalysis}, we present rigorous error bounds and their proofs for
the proposed multisketched \texttt{rand\_cholQR}. 
These bounds can also be applied to the case
of a single sketch matrix, and we compare the new results to those available
in the literature. Numerical experiments are presented in Section~\ref{numer:sec},
followed by our conclusions.
For completeness, all detailed run times of our experiments are reported in an Appendix.

\section{Related Work} \label{sec:relWork}
In the case of a single sketch matrix, the concept of sketching a tall-and-skinny matrix, computing its QR factorization, and then preconditioning the matrix with the resulting triangular factor like \texttt{randQR} (Algorithm \ref{alg:randQR}) is not new. The earliest appearance of such an algorithm was by Rokhlin and Tygert in 2008 \cite{RokhlinTygert} for solving overdetermined least squares problems, where they proposed a version of \texttt{randQR} with a column-pivoted QR factorization and a single subsampled randomized Hadamard transform sketch. Fan et al.~\cite{Fan:NovelRandomizedXR} proposed a similar algorithm called rQR-CholQR, but they only used a very simple sketch based on sampling rows of $A$. They did not consider subspace embeddings or multisketching.

Prior to this paper, Balabanov and Grigori proposed the ``RCholeskyQR" method in an unpublished manuscript \cite{GrigoriRandBlock}, which is identical to what we refer to as \texttt{randQR}, in the case of a single $(\varepsilon,d,m)$ oblivious $\ell_2$-subspace embedding. While this paper was being written, Balabanov gave stability results similar to Corollary \ref{cor:randQRcond} in an additional unpublished manuscript \cite{Balabanov:2022:cholqr}. However, our results differ from Balabanov's, as ours impose no assumptions on the level of accuracy performed by subroutines within the algorithm, and we meticulously derive all bounds from existing roundoff error analysis of each subroutine. Additionally, Balabanov's work imposes a far stricter limit on the subspace embedding parameter $\epsilon \leq \frac{1}{2}$, while ours provides analysis up to $\epsilon < \frac{616}{634}$ for a $(\varepsilon,d,m)$ oblivious $\ell_2$-subspace embedding, which is nearly the theoretical upper limit of $\epsilon < 1$ imposed by the theory in Section \ref{sec:randSketchTheory}. This is significant, because stability guarantees for larger values of $\epsilon$ ensure high accuracy with smaller sketch matrices, resulting in a more computationally efficient algorithm, as demonstrated in Section \ref{sec:epsResults}.

Our results extend beyond a single $(\varepsilon,d,m)$ oblivious $\ell_2$-subspace embedding, and cover the more generalized case of two subspace embeddings (i.e., multisketch). Also, our work includes explicit analysis of the $S_2S_1$-orthogonality error of \texttt{randQR}, which is a specific notion of orthogonality with respect to a sketched inner product, and the loss of orthogonality error in the standard Euclidean inner product of \texttt{rand\_cholQR}.

Our work is novel in several ways. To our knowledge, this work is the first to propose and analyze a randomized QR algorithm with multiple sketches. The stability results in this paper improve upon and expand the existing stability analysis of \texttt{randQR} and \texttt{rand\_cholQR}, and considers the multisketch case for the first time. Additionally, our experimental results are the first to demonstrate the performance of \texttt{rand\_cholQR} in a parallel heterogeneous computing environment under any sketching framework, particularly in the multisketch case which allows the algorithm to sometimes run faster than the widely used high-performance \texttt{cholQR2} algorithm. This tangibly demonstrates the potential of the multisketch \texttt{rand\_cholQR} in exascale applications. 

\section{Preliminaries on Random Sketching} \label{sec:randSketchTheory}
Suppose one would like to compress $V \in \mathbb{R}^{n \times m}$ into a matrix with 
fewer rows with nearly the same norm. We denote the \textit{sketch matrix} by $S \in \mathbb{R}^{p \times n}$ for $p \ll n$.
The sketch matrix is typically chosen to be a \textit{$\epsilon$-subspace embedding} (defined below), or a linear map to a lower dimensional space that preserves $\ell_2$-inner products and norms of all vectors within the subspace up to a factor of $\sqrt{1 \pm \varepsilon}$ for $\varepsilon \in [0,1)$ \cite{Balabanov:2022,Nakatsukasa:2021,Sarlos:2006}.

\begin{definition}[$\varepsilon$-subspace embedding] \label{def:subspaceEmbedding}
Given $\varepsilon \in [0,1)$, the sketch matrix \linebreak $S \in \mathbb{R}^{p \times n}$ is an \emph{$\varepsilon$-subspace embedding} for the subspace $\mathcal{V} \subset \mathbb{R}^n$ if $\forall x, y \in \mathcal{V}$, 
\[ 
    | \langle x,y \rangle - \langle Sx,Sy \rangle | \leq \varepsilon \|x\|_2\|y\|_2, \label{eq:innerProdEmbedding}
\] 
where $\langle \cdot , \cdot \rangle$ is the Euclidean inner product.
\end{definition}

\begin{proposition} {\rm \cite{Balabanov:2022} }\label{prop:embeddingNorm}
If the sketch matrix $S \in \mathbb{R}^{p \times n}$ is an \emph{$\varepsilon$-subspace embedding} for the subspace $\mathcal{V} \subset \mathbb{R}^n$, then $\forall x \in \mathcal{V}$, then,
\begin{equation}
    \sqrt{1-\varepsilon}~\|x\|_2 \leq \|Sx\|_2 \leq \sqrt{1+\varepsilon}~\|x\|_2. \label{eq:normEmbedding} 
\end{equation}
\end{proposition}
\begin{corollary} \label{cor:embeddingNorm}
    If the sketch matrix $S \in \mathbb{R}^{p \times n}$ is an \emph{$\varepsilon$-subspace embedding} for the subspace $\mathcal{V} \subset \mathbb{R}^n$, and $V$ is a matrix whose columns form a basis of $\mathcal{V}$, then
    \begin{align}
        \sqrt{1-\varepsilon} \|V\|_2 &\leq \|SV\|_2 \leq \sqrt{1+\varepsilon} \|V\|_2, \label{eq:SV2Nrmcor} \\
        \sqrt{1-\varepsilon} \|V\|_F &\leq \|SV\|_F \leq \sqrt{1+\varepsilon} \|V\|_F, \label{eq:SVfroNrmcor} \\
        \sqrt{1-\varepsilon}~\sigma_{min}(V) &\leq \sigma_{min}(SV) \leq \sqrt{1+\varepsilon}~\sigma_{max}(V). \label{eq:SVsingValcor}
    \end{align}
\end{corollary}
Proposition \ref{prop:embeddingNorm} follows simply by substituting $y = x$ in Definition \ref{def:subspaceEmbedding}, and Corollary \ref{cor:embeddingNorm} is a simple consequence of Proposition \ref{prop:embeddingNorm} using the definition of the $\ell_2$ matrix norm, the Frobenius norm, and the minimum singular value. Furthermore, Proposition \ref{prop:embeddingNorm} can be used to relate the singular values of $SV$ to those of $V$ in a different way to bound the condition number of $V$ by that of $SV$. 
\begin{proposition} {\rm \cite{Balabanov:2022} }\label{prop:embeddingSingularVals}
If the sketch matrix $S \in \mathbb{R}^{p \times n}$ is an \emph{$\varepsilon$-subspace embedding} for the subspace $\mathcal{V} \subset \mathbb{R}^n$, and $V$ is a matrix whose columns form a basis of $\mathcal{V}$, then 
    \begin{align}
        (1+\varepsilon)^{-1/2}~\sigma_{min}(SV) &\leq \sigma_{min}(V) \leq \sigma_{max}(V) 
\leq (1-\varepsilon)^{-1/2}~\sigma_{max}(SV). \label{eq:singValEmbedding} 
    \end{align}
    Thus,
    \begin{equation}
        \kappa(V) \leq \sqrt{\frac{1-\varepsilon}{1+\varepsilon}}~\kappa(SV). \label{eq:condEmbedding}
    \end{equation}
\end{proposition}

Proposition \ref{prop:embeddingSingularVals} implies that if $SV$ is well conditioned, then so is~$V$.

While $\varepsilon$-subspace embeddings require knowledge of the subspace $\mathcal{V} \subset \mathbb{R}^n$ a priori, \emph{$(\varepsilon, d, m)$ oblivious $\ell_2$-subspace embeddings} do not.

\begin{definition}[$(\varepsilon, d, m)$ oblivious $\ell_2$-subspace embedding] {\rm \cite{Balabanov:2022} }\label{def:obliviousSubspaceEmbedding}
$S \in \mathbb{R}^{p \times n}$ is an \emph{$(\varepsilon, d, m)$ oblivious $\ell_2$-subspace embedding} if it is an $\varepsilon$-subspace embedding for any fixed $m$-dimensional subspace $\mathcal{V} \subset \mathbb{R}^n$ with probability at least $1-d$.
\end{definition}

An example of a $(\varepsilon, d, m)$ oblivious $\ell_2$-subspace embedding is 
$S = \frac{1}{\sqrt{p}}G$ for a fully dense Gaussian matrix $G \in \mathbb{R}^{p \times n}$ and 
\[ 
p = \Omega(\varepsilon^{-2}\log m \log(1/d));
\] 
see, e.g., \cite{Sarlos:2006}. 
Sparse $(\varepsilon, d, m)$ oblivious $\ell_2$-subspace embeddings exist, including CountSketch, 
which consists of a single $\pm 1$ per column, where the row storing the entry and its sign are chosen uniformly at random~\cite{CountSketchOriginal,WoodruffSketching}. In order to be a $(\varepsilon, d, m)$ oblivious $\ell_2$-subspace embedding, the number of columns of the CountSketch matrix must satisfy 
\begin{equation} \label{eq:s}
p \geq \frac{m^2+m}{\varepsilon^2 d}~;
\end{equation}
see \cite{CountSketchSize}.
Other popular $(\varepsilon, d, m)$ oblivious $\ell_2$-subspace embeddings include sub-sampled randomized Hadamard and Fourier transforms, and ``sparse dimension reduction maps'' \cite{Balabanov:2022, Nakatsukasa:2021}, though obtaining high performance with these is difficult, and the complexity of applying them is higher than CountSketch. We do not consider such embeddings in this paper.

\section{Multisketching} \label{multi:sec}
Next, we consider the case of applying two sketch matrices one after the other, which is what we refer to as ``multisketching"
in this paper, generalizing the approach of 
\cite{Kapralov_Potluru_PMLR_16,Sobczyk_Gallopoulos_TOMS_22}, where
one application of a large sparse CountSketch is followed by a smaller dense Gaussian sketch. In these references though,
there is no analysis of stability, as we do here.
The main motivation for this approach is to be able to apply the dense Gaussian sketch to
a smaller matrix, obtained after the application of a sparse sketch, thus significantly reducing the size of the matrix with little computational cost;
see more details on this motivation in Section~\ref{sec:twoSketches}.

We first present the algorithm \texttt{randQR} using this multisketching approach, and then prove
bounds similar to those in Proposition~\ref{prop:embeddingSingularVals} for the case
of two sketches. We emphasize that Algorithm \ref{alg:randQR} will be used as a pre-processing procedure for the final algorithm of interest, Algorithm \ref{ajh:alg:sketchQRChol}, which will form the true QR factorization of a tall-and-skinny matrix. 

Let $V \in \mathbb{R}^{n \times m}$, and suppose $S_1 \in \mathbb{R}^{p_1 \times n}$ and $S_2 \in \mathbb{R}^{p_2 \times p_1}$ are $(\varepsilon_1, d_1, m)$ and $(\varepsilon_2, d_2, m)$ oblivious $\ell_2$-subspace embeddings, respectively. Let $d = d_1+d_2-~d_1d_2$, so that $1-d = (1-d_1)(1-d_2)$. We define the Randomized Householder QR algorithm (\texttt{randQR}) in Algorithm~\ref{alg:randQR}, where we use MATLAB function call notation. 

\begin{algorithm}[h]
         \begin{algorithmic}[1]  
             \vspace{0cm}
             \Statex \textbf{Input:}\phantom{aa} Matrix $V \in \mathbb{R}^{n \times m}$, sketch matrices $S_1 \in \mathbb{R}^{p_1 \times n}$, $S_2 \in \mathbb{R}^{p_2 \times p_1}$
            \Statex {\textbf{Output:} $S_2S_1$-Orthogonal factor $Q \in \mathbb{R}^{n \times m}$, triangular factor $R \in \mathbb{R}^{m \times m}$ {such that $QR=V$.}}
             \vspace{0cm}
             \State Apply sketches $W = S_2S_1V$ 
             \State Perform Householder QR: $[Q_{tmp},R] = \texttt{hhqr}(W)$ 
             \State Recover $S_2S_1$-orthogonal matrix: $Q = VR^{-1}$
         \end{algorithmic}
         \caption{Randomized Householder QR: $[Q,R] = \texttt{randQR}(V,S_1,S_2)$} \label{alg:randQR}
 \end{algorithm}


\begin{remark} \label{Vfullrank:rem}
{\rm
In exact arithmetic, provided that $V \in \mathbb{R}^{n \times m}$ is full rank, then \texttt{randQR} produces a matrix $Q$ that is $S_2S_1$-orthogonal\footnote{In exact arithmetic, $Q$ will only fail to be $S_2S_1$-orthogonal if $V \in \text{null}(S_2S_1)$, which by Proposition \ref{prop:embeddingSingularVals}, occurs with probability at most $d$.} with probability at least $1-d$; i.e., it satisfies $(S_2S_1Q)^T(S_2S_1Q) = I$, because 
$$S_2S_1Q = S_2S_1VR^{-1} = WR^{-1} =~Q_{tmp},$$

\noindent where $Q_{tmp}$ is the orthogonal factor produced by the Householder QR factorization of $W = S_2S_1V$. Observe that $Q$ being $S_2S_1$-orthogonal is equivalent to being an orthonormal matrix with respect to the inner product\footnote{Although $\langle S_2S_1 \cdot, S_2S_1 \cdot \rangle$ is not an inner product over the traditional vector space $\mathbb{R}^{n \times m}$, it is an inner product over the complement of $\text{null}(S_2S_1)$.} $\langle S_2S_1 \cdot, S_2S_1 \cdot \rangle$. Unlike traditional Householder QR, even in exact arithmetic $V$ must have full rank, since step~3 of Algorithm \ref{alg:randQR} requires $\text{rank}(V) = \text{rank}(R) = m$. In finite precision, intuition suggests that an inevitable requirement of \texttt{randQR} is that $V$ must be numerically full rank (i.e., $\kappa(V) \lessapprox \textbf{u}^{-1}$). 
}
\end{remark}

Next, we introduce some convenient norm, singular value, and condition number inequalities when one uses the multisketching approach with two oblivious $\ell_2$-subspace embeddings.

\begin{proposition}\label{cor:subspaceEmbeddingMultSketch}
    Let $S_1 \in \mathbb{R}^{p_1 \times n}$ be a $(\varepsilon_1, d_1, m)$ oblivious $\ell_2$-subspace embedding in $\mathbb{R}^n$, $S_2 \in \mathbb{R}^{p_2 \times p_1}$ be a $(\varepsilon_2, d_2, m)$ oblivious $\ell_2$-subspace embedding in $\mathbb{R}^{p_1}$, generated independently. 
    Let $\varepsilon_L = \varepsilon_1+\varepsilon_2-\varepsilon_1\varepsilon_2$, $\varepsilon_H = \varepsilon_1+\varepsilon_2+\varepsilon_1\varepsilon_2$, and $d = d_1+d_2-d_1d_2$. Then for any $m$-dimensional subspace $\mathcal{V}\subset \mathbb{R}^n$ and $\forall x \in \mathcal{V}$, 
    \begin{equation}
         \sqrt{1-\varepsilon_L}~\|x\|_2 \leq \|S_2S_1x\|_2 \leq \sqrt{1+\varepsilon_H}~\|x\|_2, \label{eq:normEmbeddingMultiSketch} 
    \end{equation}
    with probability at least $1-d$.
\end{proposition}

\begin{proof}
    Let $x\in \mathcal{V}$. Then, $S_1\mathcal{V} \subset \mathbb{R}^{p_1}$. By assumption, $S_2$ is a $(\varepsilon_2,d_2,m)$ oblivious $\ell_2$-subspace embedding, and thus it is an $\varepsilon_2$-subspace embedding of $S_1\mathcal{V}$ with probability at least $1-d_2$. Observe that $S_1x \in S_1\mathcal{V}$. Therefore, by \eqref{eq:normEmbedding},
\[  \sqrt{1-\varepsilon_2} \|S_1x\|_2 \\ \leq \|S_2S_1x\|_2 \leq \sqrt{1+\varepsilon_2} \|S_1x\|_2 ,\]
with probability at least $1-d_2$, because this is the probability at which \eqref{eq:normEmbedding} holds.

Again, by assumption, $S_1$ is a $(\varepsilon_1,d_1,m)$ oblivious $\ell_2$-subspace embedding, and thus it is an $\varepsilon_1$-subspace embedding of $\mathcal{V} \subset \mathbb{R}^{n}$ with probability at least $1-d_1$. Now, using \eqref{eq:normEmbedding} again for $S_1$ and $\epsilon_1$, we have
\[
\sqrt{1-\varepsilon_1}\|x\|_2 \leq \|S_1x\|_2 \leq \sqrt{1+\varepsilon_1}\|x\|_2,
\]
with probability at least $1-d_1$. 

Combining these results, we find that
\begin{align*}
    \sqrt{1-(\varepsilon_1+\varepsilon_2-\varepsilon_1\varepsilon_2)} \|x\|_2 &= \sqrt{(1-\varepsilon_2)(1-\varepsilon_1)}\|x\|_2 \leq \sqrt{1-\varepsilon_2} \|S_1x\|_2 \\ &\leq \|S_2S_1x\|_2 \leq \sqrt{1+\varepsilon_2} \|S_1x\|_2 \\
    &\leq \sqrt{(1+\varepsilon_2)(1+\varepsilon_1)}\|x\|_2 \\
    &= \sqrt{1+(\varepsilon_1+\varepsilon_2+\varepsilon_1\varepsilon_2)} \|x\|_2
\end{align*}
with probability at least $(1-d_1)(1-d_2) = 1-(d_1+d_2-d_1d_2)$. 

Proving $d$ and consequently $1-d$ are between $[0,1]$ is equivalent to showing $p(d_1,d_2) = d_1+d_2-d_1d_2 \in [0,1]$ for any $(d_1,d_2) \in~[0,1]^2$. This is straightforward, as on the boundaries, $p(0,d_2) = d_2 \in~[0,1]$, $p(d_1,0) = d_1 \in [0,1]$, \linebreak $p(1,d_2) = p(d_1,1) = 1 \in [0,1]$, and $\nabla p \geq 0$ on $[0,1]^2$, and therefore $p(d_1,d_2)$ cannot go below $0$ or above $1$.
\end{proof}

\noindent If $S_1, S_2$ are $\varepsilon_1, \varepsilon_2$ embeddings respectively, then by Corollary \ref{cor:embeddingNorm} along with Propositions \ref{prop:embeddingSingularVals} and  \ref{cor:subspaceEmbeddingMultSketch},
\begin{align}
    \sqrt{1-\varepsilon_L} \|V\|_2 &\leq \|S_2S_1V\|_2 \leq \sqrt{1+\varepsilon_H} \|V\|_2, \label{eq:S2S1V2Nrmcor} \\
        \sqrt{1-\varepsilon_L} \|V\|_F &\leq \|S_2S_1V\|_F \leq \sqrt{1+\varepsilon_H} \|V\|_F,\label{eq:S2S1VfroNrmcor} \\
        \sqrt{1-\varepsilon_L}~\sigma_{min}(V) &\leq \sigma_{min}(S_2S_1V) \leq \sqrt{1+\varepsilon_H}~\sigma_{max}(V), \label{eq:S2S1VsingValcor}
\end{align}
\vspace*{-7mm}
\begin{align}
    (1+\varepsilon_H)^{-1/2}~\sigma_{min}(S_2S_1V) &\leq \sigma_{min}(V) \leq \sigma_{max}(V) \label{eq:multSketchSingValEmbedding_0} \\
    &\leq (1-\varepsilon_L)^{-1/2}~\sigma_{max}(S_2S_1V), \nonumber 
\end{align}
and so,
\begin{equation}
    \kappa(V) \leq \sqrt{\frac{1-\varepsilon_L}{1+\varepsilon_H}}~\kappa(S_2S_1V). \label{eq:MultSketchCondEmbedding}
\end{equation} 

\begin{remark} \label{rem:condQrandQR}
    { \rm
    By Remark \ref{Vfullrank:rem}, in exact arithmetic, the $Q$ factor computed by \texttt{randQR} is $S_2S_1$-orthogonal, and therefore, by \eqref{eq:MultSketchCondEmbedding},
    \begin{equation}
    \kappa(Q) \leq \sqrt{\frac{1-\varepsilon_L}{1+\varepsilon_H}}~\kappa(S_2S_1Q) = \sqrt{\frac{1-\varepsilon_L}{1+\varepsilon_H}} = O(1)~. \label{eq:condQrandQR}
    \end{equation}
    Thus, \texttt{randQR} serves well for applications where a well-conditioned set of vectors is sufficient, or as a pre-processing algorithm for less stable orthogonalization schemes.
    }
\end{remark}

\subsection{Algorithms and Performance Considerations} \label{sec:flopsAndIdea}
We introduce the main algorithm of interest for this paper, \texttt{rand\_cholQR} \linebreak (Algorithm \ref{ajh:alg:sketchQRChol}) and compare its performance to \texttt{cholQR2}
\cite{CholeskyQR2Def} (Algorithm~\ref{ajh:alg:cholQR2} below). In this paper, \texttt{randQR} is strictly used to precondition the tall-and-skinny matrix $V$  as a pre-processing step for \texttt{rand\_cholQR}, which is a true orthogonalization scheme. 

As a proof of concept for why \texttt{rand\_cholQR} (Algorithm \ref{ajh:alg:sketchQRChol}) should be expected to form a reasonable QR factorization, in step 1, the algorithm computes a $S_2S_1$-orthogonal factor $Q_0$ from \texttt{randQR} (Algorithm \ref{alg:randQR}). By Remark \ref{rem:condQrandQR}, in exact arithmetic,
$\kappa(Q_0) = O(1)$.
In step 2 of \texttt{rand\_cholQR}, $Q$ is computed by re-orthogonalizing $Q_0$ using Cholesky QR (\texttt{cholQR}, Algorithm \ref{ajh:alg:cholQR}). Since $\kappa(Q_0) = O(1)$, one can expect the resulting $Q$ satisfies $\|Q^TQ - I\|_2 = O(\textbf{u})$ using the roundoff error analysis of \texttt{cholQR} \cite{CholeskyQR2ErrAnalysis}. 

\begin{algorithm}
         \begin{algorithmic}[1]  
             \vspace{0cm}
             \Statex \textbf{Input:}\phantom{aa} Matrix $V \in \mathbb{R}^{n \times m}$
            \Statex {\textbf{Output:} Orthogonal factor $Q \in \mathbb{R}^{n \times m}$, triangular factor $R \in \mathbb{R}^{m \times m}$ such that $QR=V$.}
             \vspace{0cm}
             \State Compute Gram matrix $G = V^TV$
             \State Perform Cholesky on $G$:  $R = \texttt{chol}(G)$, where $G = R^TR$
             \State Recover orthogonal matrix: $Q = VR^{-1}$
         \end{algorithmic}
         \caption{Cholesky QR: $[Q,R] = \texttt{cholQR}(V)$} \label{ajh:alg:cholQR}
 \end{algorithm}

\begin{algorithm}
         \begin{algorithmic}[1]  
             \vspace{0cm}
             \Statex \textbf{Input:}\phantom{aa} Matrix $V \in \mathbb{R}^{n \times m}$, sketch matrices $S_1 \in \mathbb{R}^{p_1 \times n}$, $S_2 \in \mathbb{R}^{p_2 \times p_1}$
            \Statex {\textbf{Output:} Orthogonal factor $Q \in \mathbb{R}^{n \times m}$, triangular factor $R \in \mathbb{R}^{m \times m}$ such that $QR=V$.}
             \vspace{0cm}
             \State Recover $S_2S_1$-orthogonal matrix $Q_0$: $[Q_0,R_0] = \texttt{randQR}(V,S_1,S_2)$
             \State Perform Cholesky QR on $Q_0$: $[Q,R_1] = \texttt{cholQR}(Q_0)$
             \State Return $R$: $R = R_1R_0$
         \end{algorithmic}
         \caption{Rand. Householder-Cholesky: $[Q,R]~=~\texttt{rand\_cholQR}(V,S_1,S_2)$} \label{ajh:alg:sketchQRChol}
 \end{algorithm}

Next, we examine the expected performance of \texttt{randQR} and \texttt{rand\_cholQR} by first discussing their communication costs compared to \texttt{cholQR} and \texttt{cholQR2} (Algorithm \ref{ajh:alg:cholQR2}) respectively, and then analyze their arithmetic costs. The most computationally intensive parts of \texttt{randQR} (steps 1 and 3) are nearly identical to those of \texttt{cholQR}, in the sense that both perform a product of tall and skinny matrices, followed by a triangular solve of a tall and skinny matrix. Similar to the way \texttt{cholQR} requires only one processor synchronization total to compute the Gram matrix in step 1 of Algorithm \ref{ajh:alg:cholQR}, \texttt{randQR} only requires one synchronization total to compute $W$ in step 1 of Algorithm \ref{alg:randQR} provided $m \leq p_2 \leq p_1 \ll n$ and therefore the algorithms incur the same number of processor synchronizations\footnote{Specifically, suppose one has $p$ parallel processes and $m \leq p_2 \leq p_1 \ll n$ so that \linebreak $S_2 \in \mathbb{R}^{p_2 \times p_1}$ can be stored locally on each process. One can distribute block row partitions of $V = [V_1^T, \dots, V_p^T]^T$  and block column partitions of the larger sketch $S_1 = [(S_1)_1,\dots, (S_1)_p]$ to each of the processes, along with the entire small sketch $S_2$ to each process. Then on process $k$, one computes $W_k = S_2(S_1)_kV_k$, and then one synchronizes the processes to compute $W = S_2S_1V = \sum_{k=1}^p W_k$ in a single reduction.}. Moreover, \texttt{rand\_cholQR} and \texttt{cholQR2} simply build on these algorithms, adding passes of \texttt{cholQR} to matrices of the same size for both algorithms. Thus, like \texttt{cholQR2}, \texttt{rand\_cholQR} only requires two synchronizations total.

\begin{algorithm}
         \begin{algorithmic}[1]  
             \vspace{0cm}
             \Statex \textbf{Input:}\phantom{aa} Full rank matrix $V \in \mathbb{R}^{n \times m}$
            \Statex {\textbf{Output:} Orthogonal factor $Q \in \mathbb{R}^{n \times m}$, triangular factor $R \in \mathbb{R}^{m \times m}$}
             \vspace{0cm}
             \State Perform Cholesky QR on $W$: $[Q_0,R_0] = \texttt{cholQR}(V)$
             \State Perform Cholesky QR on $Q_0$: $[Q,R_1] = \texttt{cholQR}(Q_0)$
             \State Return $R$: $R = R_1R_0$ 
         \end{algorithmic}
         \caption{CholeskyQR2: $[Q,R] = \texttt{cholQR2}(V)$} \label{ajh:alg:cholQR2}
 \end{algorithm}
 
Next, we consider the computational costs of the algorithms. The computational cost of step 2 of \texttt{randQR} (Algorithm \ref{alg:randQR}) is negligible compared to steps~1 and~3, since $W \in \mathbb{R}^{p_2 \times m}$ with $p_2 \ll n$. 
The arithmetic cost of step 1 is dependent on the type of sketch matrices used. Suppose one replaces $S_2S_1$ with a single dense Gaussian sketch matrix $S \in \mathbb{R}^{p \times n}$, which is conceptually simple, very efficient in parallel, but computationally expensive since it is very dense. Then the arithmetic cost of \texttt{randQR} and \texttt{rand\_cholQR} (in FLOPs) are:
\begin{equation*}
 \texttt{randQR}\text{ FLOPs: } \underbrace{pm(2n-1)}_\text{Sketching} + \underbrace{2pm^2-\frac{2}{3}m^3}_\text{Householder QR} + \underbrace{nm^2}_\text{Tri. solve} \approx 2nmp + nm^2.
\end{equation*}
\begin{equation*}
\texttt{rand\_cholQR} \text{ FLOPs: } \underbrace{2nmp + nm^2}_\texttt{randQR} + \underbrace{2nm^2}_\texttt{cholQR} + \underbrace{m^2(2m-1)}_\text{Matrix mult.} \approx 2nmp + 3nm^2 .
\end{equation*}
Provided that $p = O(m)$, e.g., $p \approx 2m$, then $\texttt{rand\_cholQR} \text{ FLOPs} \approx 7nm^2$. 

In contrast, CholeskyQR2 (\texttt{cholQR2}, Algorithm~\ref{ajh:alg:cholQR2}) incurs a~cost~of
\begin{equation*}
\texttt{cholQR2} \text{ FLOPs: } \underbrace{2nm^2}_\texttt{cholQR} + \underbrace{2nm^2}_\texttt{cholQR} + \underbrace{m^2(2m-1)}_\text{Matrix mult.} \approx 4nm^2 .
\end{equation*}

Thus, $\texttt{randQR}$ (using a dense Gaussian sketch) and $\texttt{cholQR}$ have about the same asymptotic arithmetic costs. Because the two algorithms have the same communication costs and \texttt{rand\_cholQR} has a slightly higher arithmetic cost, in a large scale parallel setting, one can expect \texttt{rand\_cholQR} to run  slightly slower but on the same order of runtime as \texttt{cholQR2} and scale in the same way. However, as we show in Section \ref{sec:stateKeyResults}, \texttt{rand\_cholQR} is significantly more stable with high probability.

\subsection{Motivation for Multisketching}\label{sec:twoSketches}
To motivate the use of multisketching, we first discuss a few straightforward options using a single sketch matrix. Using a single Gaussian sketch requires a dense matrix-matrix multiply with a sketch matrix $S$ of dimension $p \times n$. In addition to performing $O(nmp)$ FLOPs to apply this sketch, we need to store and load this completely dense $p \times n$ sketch matrix. As shown in Section \ref{sec:flopsAndIdea}, the time to sketch the matrix with the dense Gaussian can dominate the total factorization time for \texttt{randQR} and consequently \texttt{rand\_cholQR}. 

One can reduce the sketching cost using a sparse sketch such as a CountSketch matrix~\cite{CountSketchOriginal}. Since the CountSketch matrix has only one non-zero per column, the cost of applying the CountSketch matrix to $V \in \mathbb{R}^{n \times m}$ is only $O(nm)$, and it requires the storage of only $O(n)$ numerical values. Additionally, CountSketch can be implemented using the sparse-matrix multiple-vector multiply (SpMM), whose optimized implementation is often available on  specific architectures.  A clever implementation can exploit the fact that applying the CountSketch matrix is equivalent to adding/subtracting subsets of rows of $V$, and can therefore be parallelized well using batched BLAS-1 kernels or a highly-optimized sparse linear algebra library.  Hence, high performance implementations of CountSketch can be achieved using only readily available linear algebra libraries. However, a CountSketch matrix requires a sketch size of $p = O(m^2)$ to maintain the $\varepsilon$-embedding properties, so \texttt{randQR}/\texttt{rand\_cholQR} requires one to factorize $W \in \mathbb{R}^{O(m^2) \times m}$ with Householder QR, which incurs $O(m^4)$ FLOPs. In contrast, a sketch size of $p = O(m)$ ensures the Gaussian sketch is an $\varepsilon$-subspace embedding, meaning the cost of the Householder QR factorization is only $O(m^3)$ FLOPs. Householder QR imposes high communication costs and does not parallelize well \cite{tsqrPerf}. As a result, on current computers, it obtains much lower performance than the BLAS-3 operations like the dense matrix product (\texttt{gemm}), and these $O(m^4)$ FLOPs for Householder QR become a performance bottleneck for sufficiently large $m$. 

Ideally, we want an embedding that offers low computational and storage costs like CountSketch, while returning a sketched matrix $W \in \mathbb{R}^{p \times m}$ with $p = O(m)$ like the Gaussian sketch does, to avoid a performance bottleneck from Householder QR. This is possible by using the multisketching framework with first a sparse CountSketch and then a Gaussian sketch. To see this, suppose $S_1 \in \mathbb{R}^{p_1 \times n}$ is a CountSketch matrix with $p_1 = \frac{m^2+m}{\varepsilon_1^2d_1}$, cf.~(\ref{eq:s}), and suppose $S_2 \in \mathbb{R}^{p_2 \times p_1}$ is a Gaussian sketch where $p_2 = 2m$. 

We split the computation of $W = S_2S_1V$ into two steps: first computing $W_1 = S_1V$, then $W = S_2W_1$. Storing $S_1$ only requires $O(n)$ bytes of memory, and the sparse matrix product $W_1 = S_1 V$ costs $O(nm)$ FLOPs. The cost to compute $W = S_2W_1$ costs $O(m^4)$ FLOPs, but since the dense matrix product (\texttt{gemm}) obtains much higher performance than the Householder QR, this cost became negligible in our performance studies with a GPU. The storage of $S_2$ only requires $O(m^3)$ bytes of memory, and the Householder QR factorization of the $O(m) \times m$ matrix $W$ incurs negligble computational cost as well.

Moreover, the $O(nm+m^4)$ total FLOPs incurred by \texttt{randQR} using the multisketch framework can actually be lower than the $O(nm^2)$ FLOPs required to perform \texttt{cholQR}, making \texttt{rand\_cholQR} sometimes cheaper than \texttt{cholQR2} under the multisketch framework while incurring the same number of communications (as discussed in Section \ref{sec:flopsAndIdea}). Thus, the multisketch framework provides an avenue for an extremely efficient, stable QR factorization that can potentially outperform \texttt{cholQR2} in terms of both stability and practical speed on modern parallel machines.

\section{Error Analysis of \texttt{randQR} and \texttt{rand\_cholQR}} \label{sec:errAnalysis}

In this section we present the main results of this work on theoretical properties (with high probability) of $\hat{Q}$ and $\hat{R}$ computed by \texttt{randQR} and \texttt{rand\_cholQR}. The structure of the section is as follows. First, we highlight the sources of floating point error of the \texttt{randQR} algorithm in Section \ref{sec:errSrc}. Then, in Section \ref{appendix1}, we introduce our assumptions for the proofs and some preparatory results for the error analysis. 

We identify which results are probabilistic, and explicitly state the necessary assumptions and some useful initial consequences in Sections \ref{sec:assump}--\ref{sec:probRes}. 
In Sections \ref{sec:matMulErr}--\ref{sec:backErrRepErr}, we analyze how errors propagate through each step of \texttt{rand\_cholQR}.
Finally, in Section \ref{sec:stateKeyResults}, we provide the key theorems on the stability and accuracy of our \texttt{randQR} and \texttt{rand\_cholQR} algorithms, and prove them primarily through the preparatory results from Section \ref{appendix1}. 
Some readers may want to go directly to Section~\ref{sec:stateKeyResults} for our main results.

\subsection{Sources of Floating Point Error in \texttt{randQR}} \label{sec:errSrc}

We use a hat to denote a computed version of each of the matrix in all algorithms. First, errors are incurred when performing the matrix products $W = S_2S_1V$ in step~1 of Algorithm \ref{alg:randQR}. Specifically, there exist error terms $\Delta \hat{W}_1$, $\Delta \hat{W}$ such that
\begin{align}
\hat{W}_1 &= S_1V+\Delta \hat{W}_1, \nonumber \\
\hat{W} &= S_2\hat{W}_1 + \Delta \hat{W} ~\cdot \label{eq:hatWdef}
\end{align}
We can group these error terms together so that the computed $\hat{W}$ satisfies
\begin{align}
    \hat{W} &= S_2S_1V + E_1, \label{eq:Weq}
\end{align}
where $E_1 = S_2\Delta \hat{W}_1 + \Delta \hat{W}$. The error term $E_1$ is analyzed in Section \ref{sec:matMulErr}.

Applying Householder QR to $\hat{W}$ in step 2 incurs error $E_2$. Only the triangular factor $\hat{R}$ is needed, so some (exactly) orthogonal $Q_{tmp}$ exists such that
\begin{align}
    Q_{tmp} \hat{R} &= \hat{W} + E_2 = S_2S_1V + E_1 + E_2. \label{eq:QRrep} 
\end{align}
Analysis of $E_2$ is provided in Section \ref{sec:BackErrhhqr}.

In step 3, solving the triangular system $Q\hat{R} = V$ also creates errors. These are analyzed in a row-wise fashion in Section \ref{sec:backErrFwdSub}, taking the form
\begin{align}
    \hat{Q}_{i,:} &= V_{i,:}(\hat{R} + \Delta\hat{R}_i)^{-1} \quad (i = 1, 2, \dots m), \label{eq:RbackwardErr}
\end{align}
where $\hat{Q}_{i,:}$ and $V_{i,:}$ denote the $i^{th}$ rows of $\hat{Q}$ and $V$, respectively, and $\Delta \hat{R}_i$ is an error term incurred during the solution of the triangular systems. Finally, we recast the errors incurred in step 3 as $\hat{Q} = (V+\Delta \tilde{V})\hat{R}^{-1}$ in Sections \ref{sec:2nrmRinv}--\ref{sec:backErrRepErr}, which simplifies the analysis of the orthogonality of $\hat{Q}$ in Section \ref{sec:stateKeyResults}.

\subsection{Assumptions and Preparatory Results for our Proofs}
\label{appendix1}
Let $V \in \mathbb{R}^{n\times m}$, $n \gg m$, and suppose $S_1 \in \mathbb{R}^{p_1 \times m}$ and $S_2 \in \mathbb{R}^{p_2 \times p_1}$ are $(\varepsilon_1, d_1, m)$ and $(\varepsilon_2, d_2, m)$ oblivious $\ell_2$-subspace embeddings, respectively, generated independently. Define $d = d_1+d_2-d_1d_2$, $\varepsilon_L = \varepsilon_1+\varepsilon_2-\varepsilon_1\varepsilon_2$, $\varepsilon_H = \varepsilon_1+\varepsilon_2+\varepsilon_1\varepsilon_2$.

\subsubsection{Assumptions} \label{sec:assump}
 For the sake of organization, we define a set of assumptions stating that $V$ is sufficiently numerically full rank (i.e., $\kappa(V) \lessapprox \textbf{u}^{-1}$), $n \gg m$, and that the sketch matrices $S_1,S_2$ simultaneously satisfy the subspace embedding properties, ensuring equations \eqref{eq:normEmbedding}--\eqref{eq:condEmbedding}, \eqref{eq:normEmbeddingMultiSketch}--\eqref{eq:MultSketchCondEmbedding} hold with probability at least $1-d$. We also impose an assumption that $\epsilon_L$ is sufficiently--but need not be too far--below 1, to obtain a positive lower bound on $\sigma_m(\hat{Q})$ while maintaining as general of a result as possible.
\begin{assumption}
Suppose $S_1 \in \mathbb{R}^{p_1 \times m}$ and $S_2 \in \mathbb{R}^{p_2 \times p_1}$ are $(\varepsilon_1, d_1, m)$ and $(\varepsilon_2, d_2, m)$ oblivious $\ell_2$-subspace embeddings respectively, generated independently. Define $d = d_1+d_2-d_1d_2$, $\varepsilon_L = \varepsilon_1+\varepsilon_2-\varepsilon_1\varepsilon_2$, $\varepsilon_H = \varepsilon_1+\varepsilon_2+\varepsilon_1\varepsilon_2$, where 
$$\varepsilon_L \in \left[0,\frac{616}{625}-\frac{9}{625}\varepsilon_H \right).$$ 
Further, suppose $V \in \mathbb{R}^{n \times m}$ has full rank and $1 < m \leq p_2 \leq p_1 \leq n$ where $nm\textbf{u} \leq \frac{1}{12}$, $p_1\sqrt{p_2}\textbf{u} \leq \frac{1}{12}$, and
\begin{equation}
    \delta = \frac{383\left( \sqrt{1+\varepsilon_H}~p_2m^{3/2} + \sqrt{m}\|S_2\|_2(p_1\sqrt{p_2}\sqrt{1+\varepsilon_1}+n\|S_1\|_F)\right) }{\sqrt{1-\varepsilon_L}}\textbf{u}~\kappa(V) \leq 1 ~. \label{eq:delta}
\end{equation}
\label{assump:theoremAssumptions}
\end{assumption} 

The assumption that the integers $m, p_2, p_1, \text{ and } n$ have the ordering
\begin{equation}
    1 < m \leq p_2 \leq p_1 \leq n \label{eq:s1assumption}
\end{equation}
is logical, otherwise the embeddings $S_2 \in \mathbb{R}^{p_2 \times p_1}$ and $S_1 \in \mathbb{R}^{p_1 \times n}$ project $V$ into a larger space, defeating the purpose of sketching. Further, we assume
\begin{equation}
    nm\textbf{u} \leq \frac{1}{12} \quad \text{ and } \quad p_1\sqrt{p_2}\textbf{u} \leq \frac{1}{12}~, \label{eq:cnmuBound}
\end{equation}
which are not directly implied by \eqref{eq:delta}, but, depending of the values of $p_1$, $p_2$, $m$, $\|S_1\|_F$, and $\|S_2\|_2$, typically are consequences of it.
Note that for the usual case of $\textbf{u} = 2^{-52} \approx 10^{-16}$,
the assumption \eqref{eq:cnmuBound} is easily satisfied, e.g,. with $n=10^{12}$, $m=10^2$.

As is customary in error analysis calculations, e.g., as in 
\cite{fm1967,HighamNumAlg,wi1965}, we
define for any positive integer $k$
\begin{equation} \label{eq:gamma_k}
    \gamma_k := \frac{k\textbf{u}}{1-k\textbf{u}} ~\cdot
\end{equation}
Provided $k\textbf{u} < \frac{1}{11}$, it follows that $\gamma_k \leq 1.1k\textbf{u}$. In particular, since $\kappa(V) \geq 1$, we can deduce from \eqref{eq:delta} that 
the constants that we use in our analysis are bounded as follows,
\begin{equation}
p_2m^{3/2}\textbf{u} \leq \frac{1}{383} < \frac{1}{11}~,\label{eq:cnubounds}
\end{equation}
so \eqref{eq:s1assumption}--\eqref{eq:cnubounds} imply
\begin{align}
\gamma_n \leq 1.1n\textbf{u}, \text{ } \gamma_{m} &\leq 1.1m\textbf{u}, \text{ } \gamma_{p_1} \leq 1.1p_1\textbf{u}, \text{ } \gamma_{p_2m} \leq 1.1p_2m\textbf{u}, \nonumber \\
\gamma_{29p_2m} &\leq 31.9p_2m\textbf{u} < \frac{383}{12} p_2m\textbf{u} ~, \label{eq:gammaBounds}
\end{align}
and
\begin{equation}
    1 + \gamma_n < 1.1. \nonumber 
\end{equation}
Observe that by \eqref{eq:delta}, it follows that $383p_2m^{3/2}\textbf{u} \leq 1$. This
implies that
\begin{equation}
  \gamma_{29p_2m} \leq 31.9p_2m^{3/2}\textbf{u} < \frac{383}{12} p_2m^{3/2} \textbf{u} \leq \frac{1}{12}, \label{eq:cp2mbound} 
\end{equation}
and so
\begin{equation}
    1 + 1.1p_2m^{3/2}\textbf{u} \leq 1 + 31.9p_2m^{3/2}\textbf{u}  \leq  1 + \frac{1}{12} < 1.1 ~. \label{eq:nassumption}
\end{equation}

Finally, we will repeatedly use well-known bounds relating the $\ell_2$ and Frobenius norms,
\begin{align}
    \|X\|_2 \leq \|X \|_F &\leq \sqrt{m}\|X\|_2, \label{eq:ell2FrobNormRelation} \\
    \|XY\|_F &\leq \|X\|_2 \|Y\|_F, \quad \text{for any } X \in \mathbb{R}^{n \times m}, Y \in \mathbb{R}^{m \times k}. \label{eq:FrobNormProduct} 
\end{align}

\begin{remark}
    { \rm
    
    Note that instances of $\|S_1\|_F$ and $\|S_2\|_2$ in \eqref{eq:delta} do not dominate the bound on $\kappa(V)$ imposed by \eqref{eq:delta}.  If $S_1 \in \mathbb{R}^{p_1 \times n}$ is an unscaled CountSketch and $S_2 \in \mathbb{R}^{p_2 \times p_1}$ is a scaled Gaussian sketch, there is a deterministic bound
\[
\|S_1\|_2 \leq \|S_1\|_F \leq \sqrt{n},
\]
and there is a probabilistic bound that there is some constant $C$ such that
\[
    \|S_2\|_2 \leq 1 + C\left(\sqrt{\frac{p_1}{p_2}} + \frac{3}{\sqrt{p_2}}\right),
\]
with probability at least $1 - 2e^{-9} \approx 0.9998$ \cite{Vershynin_2018}. Thus, in the case of $p_1 = O(m^2)$ and $p_2 = O(m)$, it follows that $\|S_1\|_2 = O(\sqrt{n})$ and $\|S_2\|_2 = O(\sqrt{m})$ with very high probability. Therefore, condition \eqref{eq:delta} ultimately requires $$\delta~\leq~g(n,m,p_1,p_2)\textbf{u}~\kappa(V) \leq 1,$$ where $g$ is some low-degree polynomial for reasonable choices of sketches $S_1$, $S_2$.
}
\end{remark}

\subsubsection{Notes on Probabilistic Results} \label{sec:probRes}

While some bounds constructed throughout the proofs of our results are deterministic, several are probabilistic. Here, we address specifically which equations are not deterministic, their prerequisite assumptions, and the probabilities with which they hold.

Throughout the proofs, it is assumed that $S_1$ embeds the column space of $V$ and $S_2$ embeds the column space of $S_1V$ simultaneously, which happens with probability at least $1-d = (1-d_1)(1-d_2)$ because $S_1$ and $S_2$ are independently generated $(\varepsilon_1, d_1, m)$ and $(\varepsilon_2, d_2, m)$ oblivious $\ell_2$-subspace embeddings respectively. Therefore, by \eqref{eq:SV2Nrmcor}, \eqref{eq:SVfroNrmcor}, \eqref{eq:S2S1V2Nrmcor}--\eqref{eq:S2S1VsingValcor}, and \eqref{eq:S2S1VsingValbound},
\begin{align}
    \sqrt{1-\varepsilon_1}\|V\|_2 &\leq \|S_1V\|_2 \leq \sqrt{1 + \varepsilon_1}\|V\|_2 \nonumber \\ 
    \sqrt{1-\varepsilon_L}\|V\|_2 & \leq \|S_2S_1V\|_2 \leq \sqrt{1 + \varepsilon_H}\|V\|_2 \nonumber  \\ 
    \sqrt{1-\varepsilon_1}\|V\|_F &\leq \|S_1V\|_F \leq \sqrt{1 + \varepsilon_1}\|V\|_F \label{eq:S1VFrobound}\\
    \sqrt{1-\varepsilon_L}\|V\|_F & \leq \|S_2S_1V\|_F \leq \sqrt{1 + \varepsilon_H}\|V\|_F , \label{eq:S2S1VFrobound} \\
    \sqrt{1-\varepsilon_L}~\sigma_{min}(V) &\leq \sigma_{min}(S_2S_1V) \leq \sqrt{1+\varepsilon_H}~\sigma_{max}(V), \label{eq:S2S1VsingValbound}  
\end{align}
and
\begin{align}
    (1+\varepsilon_H)^{-1/2}~\sigma_{min}(S_2S_1V) &\leq \sigma_{min}(V) \leq \sigma_{max}(V) \label{eq:multSketchSingValEmbedding} \\
    &\leq (1-\varepsilon_L)^{-1/2}~\sigma_{max}(S_2S_1V), \nonumber 
\end{align}
along with the analogous statements for matrices whose column spaces are identical to $V$, will simultaneously hold with probability at least $1-d$. Specifically, this implies equations \eqref{eq:propOneEmbedResult}, \eqref{eq:Whatbound}, \eqref{eq:Rbound2Norm}, \eqref{eq:Rsvbound}, \eqref{eq:VRinvBound}, and \eqref{eq:VRinvSingularValBound} simultaneously hold with probability at least $1-d$, which are used to build all of the results from \eqref{eq:propOneEmbedResult}--\eqref{eq:SDelVRinvBoundNoAssump} that are prerequisite to prove Theorems \ref{thm:randQRorthpt2}--\ref{thm:randCholQRcond}, all of which hold with high probability.


\subsubsection{Forward Error in matrix-matrix multiplication $S_2S_1V$} \label{sec:matMulErr}

By \cite[Section 3.5]{HighamNumAlg}, for $A \in \mathbb{R}^{m \times n}$, $B \in \mathbb{R}^{n \times k}$, $C = AB$ executed in floating point satisfies
\begin{equation*}
    \hat{C} = AB + \Delta C, \quad |\Delta C| < \gamma_n |A||B|~,\label{eq:roundoffGemm}
\end{equation*}
where $\gamma_n$ is defined in (\ref{eq:gamma_k}).
Thus, in floating point, step 1 of Algorithm \ref{alg:randQR} becomes:
\begin{align}
\hat{W}_1 &= S_1V+\Delta \hat{W}_1, \quad |\Delta \hat{W}_1| < \gamma_n |S_1||V| ~,\label{eq:hatW1bound}\\
\hat{W} &= S_2\hat{W}_1 + \Delta \hat{W}, \quad |\Delta \hat{W}| < \gamma_{p_1} |S_2||\hat{W_1}| = \gamma_{p_1} |S_2||S_1V + \Delta \hat{W}_1| ~\cdot \label{eq:hatWbound}
\end{align}
In other words,
\begin{equation}
    \hat{W} = S_2S_1V + E_1, \label{eq:WhatDef}
\end{equation}
where the forward error of these matrix-matrix products $E_!$ is defined as:
\begin{equation*}
    E_1 = \hat{W} - W = \Delta \hat{W} + S_2\Delta \hat{W}_1~\cdot
\end{equation*}
    
By \eqref{eq:ell2FrobNormRelation}, \eqref{eq:FrobNormProduct}, \eqref{eq:S1VFrobound}, \eqref{eq:hatW1bound}, and \eqref{eq:hatWbound},
\begin{align*}
    \|E_1\|_2 &\leq \|\Delta \hat{W}\|_2 + \|S_2\|_2 \|\Delta \hat{W}_1\|_2 \\
    &\leq \|\Delta \hat{W}\|_F + \|S_2\|_2 \|\Delta \hat{W}_1\|_F \\
    &\leq \gamma_{p_1}\|S_2\|_F\|S_1V + \Delta \hat{W}_1\|_F +  \|S_2\|_2 \|\Delta \hat{W}_1\|_F \\
    &\leq \gamma_{p_1}\|S_2\|_F (\|S_1V\|_F + \|\Delta \hat{W}_1\|_F) + \|S_2\|_2 \|\Delta \hat{W}_1\|_F\\
    &\leq \gamma_{p_1}\|S_2\|_F (\sqrt{1+\varepsilon_1}\|V\|_F + \|\Delta \hat{W}_1\|_F) + \|S_2\|_2 \|\Delta \hat{W}_1\|_F \\
    &\leq \sqrt{p_2}\gamma_{p_1}\|S_2\|_2 (\sqrt{1+\varepsilon_1}\|V\|_F + \|\Delta \hat{W}_1\|_F) + \|S_2\|_2 \|\Delta \hat{W}_1\|_F \\
    &= \sqrt{p_2}\gamma_{p_1}\|S_2\|_2 \sqrt{1+\varepsilon_1}\|V\|_F + \|S_2\|_2 (1+\sqrt{p_2}\gamma_{p_1})\|\Delta \hat{W}_1\|_F \\
    &\leq \sqrt{p_2}\gamma_{p_1}\|S_2\|_2 \sqrt{1+\varepsilon_1}\|V\|_F + \|S_2\|_2 (1+\sqrt{p_2}\gamma_{p_1})\gamma_n \|S_1\|_F \|V\|_F\\
    &= \|S_2\|_2 \left(\sqrt{p_2}\gamma_{p_1}\sqrt{1+\varepsilon_1} +  \gamma_n(1+\sqrt{p_2}\gamma_{p_1}) \|S_1\|_F\right) \|V\|_F\\
    &\leq \sqrt{m}\|S_2\|_2 \left(\sqrt{p_2}\gamma_{p_1}\sqrt{1+\varepsilon_1} +  \gamma_n(1+\sqrt{p_2}\gamma_{p_1}) \|S_1\|_F\right) \|V\|_2.
    \end{align*}

\noindent Notice \eqref{eq:cnmuBound} and \eqref{eq:gammaBounds} imply $\gamma_{n}(1+\sqrt{p_2}\gamma_{p_1}) < 1.21n\textbf{u}$. Hence,

\begin{equation}
    \|E_1\|_2 \leq \sqrt{m}\textbf{u}\|S_2\|_2(1.1p_1\sqrt{p_2}\sqrt{1+\varepsilon_1}+1.21n\|S_1\|_F)\|V\|_2.~\label{eq:propOneEmbedResult}
\end{equation}

\begin{lemma} \label{lem:delVal1}
     If $S_1$ is a $\varepsilon_1$ embedding of the column space of $V$ and $S_2$ is a $\varepsilon_2$ embedding of the column space of $S_1V$, then
    \begin{equation} \label{eq:delBound1}
        \frac{12}{\sqrt{1-\varepsilon_L}}\left( 31.9\sqrt{1+\varepsilon_H}~p_2m^{3/2}\textbf{u}~\kappa(V) + 1.1\|E_1\|_2~\sigma_m(V)^{-1} \right) \leq \delta \leq 1~.
    \end{equation}
\end{lemma}
\begin{proof}
    Follows directly from using \eqref{eq:propOneEmbedResult} and the definition of $\delta$ in \eqref{eq:delta}.
\end{proof}


\subsubsection{Backward Error of Householder QR of $\hat{W}$} \label{sec:BackErrhhqr}

By \cite[Theorem 19.4]{HighamNumAlg}, Householder QR of $\hat{W} \in \mathbb{R}^{p_2 \times m}$ returns a triangular \linebreak $\hat{R} \in \mathbb{R}^{m \times m}$ so that some orthogonal $Q_{tmp} \in \mathbb{R}^{p_2 \times m}$ satisfies,
\begin{equation}
    \hat{W} + E_2 = Q_{tmp}\hat{R}, \quad \|(E_2)_j\|_2 \leq \gamma_{29p_2m}\|\hat{w}_j\|_2, \text{ for } j = 1, \dots, m .\label{eq:QRfact}
\end{equation}

We mention that in \cite{HighamNumAlg}, the bound in \eqref{eq:QRfact} has $\gamma_{cp_2m}$, for some
small integer constant $c$. It is mentioned there that the exact value of $c$ is ``unimportant" for the
general analysis. A careful look at the proof of \cite[Theorem 19.4]{HighamNumAlg} indicates that
one can take $c=29$, and this is what we have done.

By \eqref{eq:ell2FrobNormRelation}, \eqref{eq:WhatDef}, and the embedding properties of $S_2S_1$ on $V$ given in \eqref{eq:S2S1VFrobound},
\begin{align}
     \|\hat{W}\|_F &\leq \|S_2S_1V\|_F + \|E_1\|_F \leq \sqrt{1+\epsilon_H} \|V\|_F + \|E_1\|_F \nonumber \\
     &\leq \sqrt{m} \left( \sqrt{1+\epsilon_H} \|V\|_2 + \|E_1\|_2 \right), \label{eq:Whatbound}
\end{align}
and therefore by \eqref{eq:QRfact},
\begin{equation}
    \|E_2\|_F \leq \gamma_{29p_2m}  \|\hat{W}\|_F  \leq 
    \gamma_{29p_2m} \sqrt{m} \left( \sqrt{1+\epsilon_H} \|V\|_2 + \|E_1\|_2 \right)~,
    \label{eq:E2frobound}
\end{equation}
with probability at least $1-d$. Finally, by \eqref{eq:gammaBounds} and \eqref{eq:cp2mbound},
\begin{align}
    \|E_2\|_2 &\leq \|E_2\|_F \leq 31.9 p_2m\textbf{u} (\sqrt{1+\epsilon_H}\|V\|_F + \|E_1\|_F) \nonumber \\
    &\leq 31.9p_2m^{3/2}\textbf{u} (\sqrt{1+\epsilon_H}\|V\|_2 + \|E_1\|_2) \nonumber \\
    &= 31.9p_2m^{3/2}\textbf{u} \sqrt{1+\epsilon_H}\|V\|_2 + 31.9p_2m^{3/2}\textbf{u} \|E_1\|_2 \nonumber \\
     &\leq 31.9p_2m^{3/2}\textbf{u} \sqrt{1+\epsilon_H}\|V\|_2 + 0.1 \|E_1\|_2 ~,
    \label{eq:E2bound}
\end{align}
with probability at least $1-d$.

\subsubsection{Backward Error of the Forward Substitution} \label{sec:backErrFwdSub}

In Step 3 of \texttt{randQR}, we solve for $Q$ via the triangular system $Q\hat{R} = V$. By \cite[Theorem 8.5]{HighamNumAlg}, in floating point, $\hat{Q}_{i,:}$ satisfies
\begin{equation}
\hat{Q}_{i,:}(\hat{R}+\Delta R_i) = V_{i,:}, \quad |\Delta R_i| < \gamma_m |\hat{R}| \text{ for } i = 1, \dots n. \label{eq:triSys}
\end{equation}
While it would be convenient to simply write $\hat{Q}(R + \Delta R) = V$ for some $\Delta R$, each $\Delta R_i$ error incurred depends on each right hand side of \eqref{eq:triSys}, and therefore each row must be accounted for separately. For each $i = 1, \dots, n$, 
\begin{equation}
    \|\Delta \hat{R}_i \|_2 \leq \|\Delta \hat{R}_i\|_F = \| | \Delta \hat{R}_i| \|_F < \gamma_m \| | \hat{R} |\|_F = \gamma_m \|  \hat{R} \|_F   ~\cdot  \label{eq:delRprelimBound}
\end{equation}

By \eqref{eq:QRfact}, \eqref{eq:E2frobound}, and the orthogonality of $Q_{tmp}$, it follows that
\begin{align}
    \|\hat{R}\|_F &= \|Q_{tmp}\hat{R}\|_F = \| \hat{W} + E_2\|_F \leq (1+\gamma_{29p_2m}) \| \hat{W} \|_F,
    \label{eq:Rbound} \\
    \|\hat{R}\|_2 &= \| \hat{W} + E_2\|_2 = \| S_2S_1V + E_1 + E_2\|_2 \nonumber \\
    &\leq \sqrt{1+\varepsilon_H}\|V\|_2 + \|E_1\|_2 + \|E_2\|_2. \label{eq:Rbound2Norm}
\end{align}

Therefore, by \eqref{eq:cnubounds}--\eqref{eq:nassumption}, \eqref{eq:Whatbound}, \eqref{eq:delRprelimBound}, and \eqref{eq:Rbound},
\begin{align*}
     \|\Delta \hat{R}_i \|_2 &\leq 1.1m^{3/2}\textbf{u}(1+31.9p_2m\textbf{u}) \left( \sqrt{1+\epsilon_H} \|V\|_2 + \|E_1\|_2 \right) \\
     &\leq 1.21m^{3/2}\textbf{u} \left( \sqrt{1+\epsilon_H} \|V\|_2 + \|E_1\|_2 \right) \cdot
\end{align*}

\subsubsection{Bounding the 2-norm of $\hat{R}^{-1}$ and $V\hat{R}^{-1}$} \label{sec:2nrmRinv}

By \eqref{eq:S2S1VsingValbound}, \eqref{eq:Weq}, and Weyl's inequality \cite{Weyl1912}, with probability at least $1-d$,
\begin{align}
    \sigma_m(\hat{W}+E_2) &\geq \sigma_m(\hat{W}) - \|E_2\|_2 \geq \sigma_m(S_2S_1V) - (\|E_1\|_2+\|E_2\|_2) \nonumber \\
    &\geq \sqrt{1-\epsilon_L}~\sigma_m(V) - (\|E_1\|_2+\|E_2\|_2).
    \label{eq:Rsvbound}
\end{align}

\noindent By Lemma \ref{lem:delVal1}, and the fact that the fact that $\|V\|_2 = \kappa(V)~\sigma_m(V)$, 
\begin{equation}
31.9p_2m^{3/2}\textbf{u}\sqrt{1+\varepsilon_H}\|V\|_2 + 1.1\|E_1\|_2 \leq \frac{\sqrt{1-\epsilon_L}}{12}\sigma_m(V)~\delta.
\label{eq:preE1plusE2bound}
\end{equation}


\noindent Combining \eqref{eq:E2bound} and \eqref{eq:preE1plusE2bound} and the assumption that $\delta \leq 1$, results in:
\begin{align}
    \|E_1\|_2 + \|E_2\|_2 
    &\leq 31.9p_2m^{3/2}\textbf{u}\sqrt{1+\varepsilon_H}\|V\|_2 + 1.1\|E_1\|_2 \nonumber \\
    &\leq \frac{\sqrt{1-\epsilon_L}}{12} \sigma_m(V) \delta \leq \frac{\sqrt{1-\epsilon_L}}{12} \sigma_m(V),
    \label{eq:E1plusE2bound}
\end{align}
so by \eqref{eq:Rsvbound} and \eqref{eq:E1plusE2bound},
\begin{equation}
    \sigma_m(\hat{R}) = \sigma_m(Q_{tmp}\hat{R}) = \sigma_m(\hat{W}+E_2) \geq \frac{11\sqrt{1-\epsilon_L}}{12}~\sigma_m(V).
    \label{eq:RsvBound2}
\end{equation}

\noindent
Therefore, by \eqref{eq:RsvBound2}
\begin{equation}
 \|\hat{R}^{-1}\|_2 \leq \frac{12}{11\sqrt{1-\epsilon_L}} \left( \sigma_m(V) \right)^{-1}.
 \label{eq:RinvNrmBound}
\end{equation}
\noindent
By \eqref{eq:QRrep}, we have that Step 2 of \texttt{randQR} satisfies
\begin{equation}
    S_2S_1V\hat{R}^{-1} = Q_{tmp} - (E_1 + E_2)\hat{R}^{-1}.
    \label{eq:SVRinv}
\end{equation}
\noindent
Thus, by \eqref{eq:E1plusE2bound}, \eqref{eq:RinvNrmBound}, \eqref{eq:SVRinv}, the fact that $Q_{tmp}$ is orthogonal,
\begin{equation}
    \|S_2S_1V\hat{R}^{-1}\|_2 \leq \|Q_{tmp}\|_2 + (\|E_1\|_2 + \|E_2\|_2) \|\hat{R}^{-1}\|_2 \leq \frac{12}{11} ~.
    \label{eq:SVRinvbound}
\end{equation}
 Observe that $V$ and $V\hat{R}^{-1}$ have the same column space; therefore if $S_1$, $S_2$ embed the column space of $V$, they will also embed the column space of $V\hat{R}^{-1}$. Therefore, by \eqref{eq:normEmbeddingMultiSketch},
\begin{equation}
    \|V\hat{R}^{-1}\|_2 \leq \frac{1}{\sqrt{1-\epsilon_L}}\|S_2S_1V\hat{R}^{-1}\|_2  \leq \frac{12}{11\sqrt{1-\epsilon_L}} ~,
    \label{eq:VRinvBound}
\end{equation} with probability at least $1-d$.

\subsubsection{Evaluation of the Backward Error $\Delta \tilde{V} = \hat{Q}\hat{R} - V$} \label{sec:backErrRepErr}
Instead of using backward errors $\Delta R_i$ for each triangular solve in equation \eqref{eq:RbackwardErr}, we capture the errors of each triangular solve in a matrix $\Delta \tilde{V}$, where 
\begin{equation}
\hat{Q} = (V + \Delta \tilde{V} )\hat{R}^{-1} \iff \hat{Q}\hat{R} = V + \Delta \tilde{V}. \label{eq:hatQexp}
\end{equation}

\noindent From \eqref{eq:RbackwardErr}, we have $\hat{Q}_{i,:}(R+\Delta R_i) = V_{i,:}$.
Then $\Delta \tilde{V}$ can be defined row-wise,
\begin{equation}
\Delta \tilde{V}_{i,:} = - \hat{Q}_{i,:}\Delta R_i. \label{eq:delVdef}
\end{equation}
Thus, by \eqref{eq:triSys}, $|\Delta \tilde{V}_{i,:}| \leq 1.1m\textbf{u}|\hat{Q}_{i,:}| |\hat{R}|, $ and so 
$|\Delta \tilde{V}| \leq 1.1m\textbf{u}|\hat{Q}| |\hat{R}|$, 
hence,
\begin{equation*}
|\Delta \tilde{V}_{:,i}| \leq 1.1m\textbf{u}|\hat{Q}| |\hat{R}_{:,i}|.
\end{equation*}

From this, it follows that for each column $i = 1, \dots, m$, 
\begin{align*}
    \| \Delta \tilde{V}_{:,i}\|_2 &\leq 1.1m\textbf{u} \||\hat{Q}|\|_2 \| |\hat{R}_{:,i}|\|_2 \leq 1.1m\textbf{u} \||\hat{Q}|\|_F \| |\hat{R}_{:,i}|\|_2 \\
    &= 1.1m\textbf{u} \|\hat{Q}\|_F \| \hat{R}_{:,i}\|_2 \leq 1.1m^{3/2}\textbf{u} \|\hat{Q}\|_2 \| \hat{R}_{:,i}\|_2,
\end{align*}
and therefore by \eqref{eq:cp2mbound}, \eqref{eq:Whatbound}, and \eqref{eq:Rbound},
\begin{align}
    \|\Delta \tilde{V}\|_2 &\leq \|\Delta \tilde{V}\|_F \leq 1.1m^{3/2}\textbf{u} \|\hat{Q}\|_2 \| \hat{R}\|_F \nonumber \\
    &\leq 1.1m^{3/2}\textbf{u} \|\hat{Q}\|_2 (1+\gamma_{29p_2m})\|\hat{W}\|_F \nonumber \\
    &\leq 1.1m^{3/2}\textbf{u} \|\hat{Q}\|_2 (1+\gamma_{29p_2m})\sqrt{m}\left( \sqrt{1+\varepsilon_H} \|V\|_2 + \|E_1\|_2\right) \nonumber \\
    &\leq 1.21m^{2}\textbf{u} \|\hat{Q}\|_2 \left( \sqrt{1+\varepsilon_H} \|V\|_2 + \|E_1\|_2\right)~.
    \label{eq:delVboundInitial}
\end{align}

\noindent By \eqref{eq:s1assumption} it follows that $1.1\sqrt{m} \leq p_2$, and so by \eqref{eq:nassumption}, \eqref{eq:E1plusE2bound}, and \eqref{eq:delVboundInitial},
\begin{align}
    \|\Delta \tilde{V}\|_F &\leq 1.21m^{2}\textbf{u} \|\hat{Q}\|_2 \left( \sqrt{1+\varepsilon_H} \|V\|_2 + \|E_1\|_2\right)  \nonumber \\
    &= \|\hat{Q}\|_2 \left(1.21m^{2}\textbf{u}\sqrt{1+\varepsilon_H} \|V\|_2 +  1.21m^{2}\textbf{u} \|E_1\|_2\right) \nonumber \\
    &\leq \|\hat{Q}\|_2 \left(1.1p_2m^{3/2}\textbf{u}\sqrt{1+\varepsilon_H} \|V\|_2 +  1.1p_2m^{3/2}\textbf{u} \|E_1\|_2\right) \nonumber \\
    &\leq \|\hat{Q}\|_2 \left(1.1p_2m^{3/2}\textbf{u}\sqrt{1+\varepsilon_H} \|V\|_2 +  (1+1.1p_2m^{3/2}\textbf{u}) \|E_1\|_2\right) \nonumber \\
    &\leq \|\hat{Q}\|_2 \left(1.1p_2m^{3/2}\textbf{u}\sqrt{1+\varepsilon_H} \|V\|_2 +  1.1 \|E_1\|_2\right) \nonumber \\
    &\leq \|\hat{Q}\|_2 \left(31.9p_2m^{3/2}\textbf{u}\sqrt{1+\varepsilon_H} \|V\|_2 +  1.1 \|E_1\|_2\right) \nonumber \\
    &\leq \|\hat{Q}\|_2 \frac{\sqrt{1-\epsilon_L}}{12} \sigma_m(V)~\delta~\cdot \label{eq:delVboundInitial2}
\end{align}

The remaining issue to resolve is that the bound  on $\|\Delta \tilde{V}\|_F$ in \eqref{eq:delVboundInitial2} requires knowledge of $\|\hat{Q}\|_2$, which we have not yet found. Combining \eqref{eq:delta}, \eqref{eq:RinvNrmBound}, and \eqref{eq:delVboundInitial2} gives, 
\begin{align}
    \|\hat{Q} - V\hat{R}^{-1}\|_F &= \|\Delta \tilde{V}\hat{R}^{-1}\|_F \leq \|\Delta \tilde{V}\|_F\|\hat{R}^{-1}\|_2 \nonumber\\
    &\hspace*{-8mm}
    \leq \|\hat{Q}\|_2 \frac{\sqrt{1-\epsilon_L}}{12} \sigma_m(V) \delta \frac{12}{11\sqrt{1-\epsilon_L}} \left( \sigma_m(V) \right)^{-1} \nonumber\\
    &\hspace*{-8mm} = \frac{\delta}{11} \|\hat{Q}\|_2 \leq \frac{1}{11} \|\hat{Q}\|_2 
\cdot 
\label{eq:delVRinvBound}
\end{align}
\noindent
Now, by \eqref{eq:E1plusE2bound}, \eqref{eq:RinvNrmBound}, and \eqref{eq:SVRinv}, 
\begin{align}
        \|S_2S_1V\hat{R}^{-1}-Q_{tmp}\|_2  
        &\leq \left(\|E_1\|_2 + \|E_2\|_2 \right)\|\hat{R}^{-1}\|_2 \leq \frac{\delta}{11} ~\cdot\label{eq:SVRinvminusQtmp}
\end{align}
Applying Weyl's inequality to \eqref{eq:SVRinvminusQtmp} and the fact that $Q_{tmp}$ is orthogonal yields,
\begin{equation*}
    1-\frac{\delta}{11} \leq \sigma_m(S_2S_1V\hat{R}^{-1}) \leq \sigma_1(S_2S_1V\hat{R}^{-1}) \leq 1 + \frac{\delta}{11} ~\cdot 
\end{equation*}
Since $V$ and $V\hat{R}^{-1}$ have identical column spaces and $S_1$,$S_2$ embed the column space of $V$, the embedding properties in  \eqref{eq:multSketchSingValEmbedding} also apply to $V\hat{R}^{-1}$, and so
\begin{equation}
    \frac{1-\frac{\delta}{11}}{\sqrt{1+\epsilon_H}} \leq \sigma_m(V\hat{R}^{-1}) \leq \sigma_1(V\hat{R}^{-1}) \leq \frac{1 + \frac{\delta}{11}}{\sqrt{1-\epsilon_L}} \leq \frac{12}{11\sqrt{1-\epsilon_L}} ~\cdot \label{eq:VRinvSingularValBound}
\end{equation}
Then, we can use Weyl's inequality again on $\hat{Q} - V\hat{R}^{-1}$. In particular,
\begin{equation}
    \sigma_m(V\hat{R}^{-1}) - \|\hat{Q}-V\hat{R}^{-1}\|_2 \leq \sigma_m(\hat{Q}) \leq \sigma_1(\hat{Q}) \leq \sigma_1(V\hat{R}^{-1}) + \|\hat{Q}-V\hat{R}^{-1}\|_2 ~\cdot \label{eq:singValBoundQ}
\end{equation}
Then, by \eqref{eq:delVRinvBound}, \eqref{eq:VRinvSingularValBound}, and \eqref{eq:singValBoundQ},
\begin{equation}
    \|\hat{Q}\|_2 = \sigma_1(\hat{Q}) \leq \sigma_1(V\hat{R}^{-1}) + \|\hat{Q}-V\hat{R}^{-1}\|_2 \leq \frac{12}{11\sqrt{1-\epsilon_L}} + \frac{1}{11} \|\hat{Q}\|_2,
\end{equation}
and thus,
\begin{equation}
    \|\hat{Q}\|_2 \leq \frac{6}{5\sqrt{1-\epsilon_L}} ~\cdot
    \label{eq:QnrmRoughBound}
\end{equation}
Then, we obtain from \eqref{eq:delVRinvBound},
\begin{equation}
   \|\hat{Q} - V\hat{R}^{-1}\|_2 = \|\Delta \tilde{V} \hat{R}^{-1}\|_2 \leq \|\Delta \tilde{V} \hat{R}^{-1}\|_F \leq \frac{\delta}{11}\|\hat{Q}\|_2 \leq \frac{6\delta}{55\sqrt{1-\epsilon_L}}. \label{eq:delVRinvfinalbound}
\end{equation}

\subsubsection{Bounding $\|S_2S_1\Delta \tilde{V}\hat{R}^{-1}\|_2$} \label{sec:S2S1DelVRinvBound}

If no additional assumptions on the embedding of $S_1$, $S_2$ are made, clearly it follows that 
\begin{equation}
    \|S_2S_1\Delta \tilde{V} \hat{R}^{-1}\|_2 \leq \|S_2\|_2\|S_1\|_2\|\Delta \tilde{V} \hat{R}^{-1}\|_2 \leq  \frac{6\|S_2\|_2\|S_1\|_2}{55\sqrt{1-\epsilon_L}}~\delta. \label{eq:SDelVRinvBoundNoAssump}
\end{equation}

Alternatively, if we assume $S_1$, $S_2$ embed $\Delta \tilde{V} \hat{R}^{-1}$, by \eqref{eq:normEmbeddingMultiSketch},
\begin{equation}
    \|S_2S_1\Delta \tilde{V} \hat{R}^{-1}\|_2 \leq \sqrt{1+\epsilon_H}\|\Delta \tilde{V} \hat{R}^{-1}\|_2 \leq  \frac{6\sqrt{1+\epsilon_H}}{55\sqrt{1-\epsilon_L}}~\delta. \label{eq:SDelVRinvBoundAssump}
\end{equation}

\subsection{Key Theoretical Results}\label{sec:stateKeyResults}

We begin by re-stating the assumptions of the theoretical results for readability.


\begin{customthm}{5.1}
Suppose $S_1 \in \mathbb{R}^{p_1 \times m}$ and $S_2 \in \mathbb{R}^{p_2 \times p_1}$ are $(\varepsilon_1, d_1, m)$ and $(\varepsilon_2, d_2, m)$ oblivious $\ell_2$-subspace embeddings respectively, generated independently. Define $d = d_1+d_2-d_1d_2$, $\varepsilon_L = \varepsilon_1+\varepsilon_2-\varepsilon_1\varepsilon_2$, $\varepsilon_H = \varepsilon_1+\varepsilon_2+\varepsilon_1\varepsilon_2$, where 
$$\varepsilon_L \in \left[0,\frac{616}{625}-\frac{9}{625}\varepsilon_H \right).$$ 
Further, suppose $V \in \mathbb{R}^{n \times m}$ has full rank and $1 < m \leq p_2 \leq p_1 \leq n$ where $nm\textbf{u} \leq \frac{1}{12}$, $p_1\sqrt{p_2}\textbf{u} \leq \frac{1}{12}$, and
\begin{equation*}
    \delta = \frac{383\left( \sqrt{1+\varepsilon_H}~p_2m^{3/2} + \sqrt{m}\|S_2\|_2(p_1\sqrt{p_2}\sqrt{1+\varepsilon_1}+n\|S_1\|_F)\right) }{\sqrt{1-\varepsilon_L}}\textbf{u}~\kappa(V) \leq 1 ~. 
\end{equation*}

\end{customthm} 

Remark~\ref{Vfullrank:rem} indicates that in exact arithmetic, \texttt{randQR} yields a matrix $Q$ that is orthogonal with respect to $\langle S_2S_1\cdot, S_2S_1\cdot \rangle$. We show next that provided $V$ has full numerical rank, then in floating point arithmetic, the orthogonality error of the matrix $\hat{Q}$ generated by \texttt{randQR} measured in $\langle S_2S_1\cdot, S_2S_1\cdot \rangle$ is $O(\textbf{u})\kappa(V)$, and the factorization error is $O(\textbf{u})\|V\|_2$ with high probability. 

\begin{theorem}[\texttt{randQR} Errors] 
Suppose Assumptions \ref{assump:theoremAssumptions} are satisfied. 
Then the $\hat{Q}, \hat{R}$ factors obtained with Algorithm \ref{alg:randQR} (\texttt{randQR}) satisfy
\begin{equation}
    \|V - \hat{Q}\hat{R}\|_2 \leq \frac{\delta}{10}\sigma_m(V),
    \label{eq:repErr2}
\end{equation}
and
\begin{align}
        \|(S_2S_1\hat{Q})^T(S_2S_1\hat{Q}) - I\|_2 & \nonumber \\
        & \hspace{-30mm}
        \leq 2\frac{\delta}{11}+\left(\frac{\delta}{11}\right)^2 + \frac{24}{11}\frac{6\|S_2\|_2\|S_1\|_2}{55\sqrt{1-\epsilon_L}}~\delta +\left(  \frac{6\|S_2\|_2\|S_1\|_2}{55\sqrt{1-\epsilon_L}}~\delta\right)^2. \label{eq:orthErrProb1md}
  \end{align}
  with probability at least $1-d$, where $\delta$ is defined in \eqref{eq:delta}. Furthermore, 
  \begin{align}
        \|(S_2S_1\hat{Q})^T(S_2S_1\hat{Q}) - I\|_2 &\leq 3\delta \label{eq:orthErrProb1mdsq}
  \end{align}
  with probability at least $(1-d)^2$.
\label{thm:randQRorthpt2}
\end{theorem}

\begin{proof}
    Equation \eqref{eq:repErr2} follows by combining \eqref{eq:delVboundInitial2} and \eqref{eq:QnrmRoughBound}, since $\Delta \tilde{V} = \hat{Q}\hat{R} - V$, and this holds with probability at least $1-d$ because 
    \eqref{eq:delVboundInitial2} and \eqref{eq:QnrmRoughBound} hold with this probability, as discussed in Section \ref{sec:probRes}.

Observe that by \eqref{eq:QRrep}, we have $S_2S_1V = Q_{tmp}\hat{R}-(E_1+E_2)$, and thus
\begin{align*}
    (S_2S_1V)^T(S_2S_1V) &= \\
    &\hspace*{-30mm} 
    \hat{R}^T\hat{R} - (E_1+E_2)^TQ_{tmp}\hat{R} - \hat{R}^TQ_{tmp}^T(E_1+E_2) + (E_1+E_2)^T(E_1+E_2).
\end{align*}
\noindent
Using \eqref{eq:hatQexp} to expand $S_2S_1\hat{Q} = (S_2S_1V+S_2S_1\Delta \tilde{V})\hat{R}^{-1}$, we obtain
    \begin{align*}
        (S_2S_1\hat{Q})^T(S_2S_1\hat{Q})
        &= I - \hat{R}^{-T}(E_1+E_2)^TQ_{tmp} - Q_{tmp}^T(E_1+E_2)\hat{R}^{-1} \\
         &\hspace*{-25mm} 
         \phantom{=}+\hat{R}^{-T}(E_1+E_2)^T(E_1+E_2)\hat{R}^{-1} + (S_2S_1\Delta \tilde{V}\hat{R}^{-1})^T(S_2S_1V\hat{R}^{-1}) \\
        &\hspace*{-25mm} 
        \phantom{=}+ (S_2S_1V\hat{R}^{-1})^TS_2S_1\Delta \tilde{V}\hat{R}^{-1} + (S_2S_1\Delta \tilde{V}\hat{R}^{-1})^T(S_2S_1\Delta \tilde{V}\hat{R}^{-1}).
    \end{align*}
\noindent
Therefore, by \eqref{eq:E1plusE2bound}, \eqref{eq:RinvNrmBound}, and \eqref{eq:SVRinvbound},
    \begin{align}
        \|(S_2S_1\hat{Q})^T(S_2S_1\hat{Q}) - I\|_2 & \nonumber \\
        & \hspace{-20mm} \nonumber
        \leq 2(\|E_1\|_2+\|E_2\|_2)\|\hat{R}^{-1}\|_2 + (\|E_1\|_2+\|E_2\|_2)^2\|\hat{R}^{-1}\|_2^2\\
        & \hspace{-20mm}
        + 2\|S_2S_1\Delta \tilde{V}\hat{R}^{-1}\|_2\|S_2S_1V\hat{R}^{-1}\|_2 
        +\|S_2S_1\Delta \tilde{V}\hat{R}^{-1}\|_2^2 \nonumber \\
        & \hspace{-20mm}
        \leq 2\frac{\delta}{11}+\left(\frac{\delta}{11}\right)^2 + \frac{24}{11}\|S_2S_1\Delta \tilde{V}\hat{R}^{-1}\|_2 
        +\|S_2S_1\Delta \tilde{V}\hat{R}^{-1}\|_2^2. \label{eq:intermediateBoundOrth}
    \end{align}

Observe that 
\eqref{eq:E1plusE2bound}, \eqref{eq:RinvNrmBound}, and \eqref{eq:SVRinvbound}, simultaneously hold with probability at least $1-d$, because they rely on $V$  and $S_1V$ being simultaneously embedded by $S_1$ and $S_2$ respectively, which with this probability occurs, as discussed in Section \ref{sec:probRes}. Thus, \eqref{eq:intermediateBoundOrth} holds with probability at least $1-d$. Observe that
\eqref{eq:SDelVRinvBoundNoAssump} requires no further assumptions on the embedding properties of $S_1,S_2$ and so applying \eqref{eq:SDelVRinvBoundNoAssump} to \eqref{eq:intermediateBoundOrth} gives
  \begin{align}
        \|(S_2S_1\hat{Q})^T(S_2S_1\hat{Q}) - I\|_2 & \nonumber \\
    & \hspace{-20mm}
        \leq 2\frac{\delta}{11}+\left(\frac{\delta}{11}\right)^2 + \frac{24}{11}\|S_2S_1\Delta \tilde{V}\hat{R}^{-1}\|_2 
        +\|S_2S_1\Delta \tilde{V}\hat{R}^{-1}\|_2^2 \nonumber \\
        & \hspace{-20mm}
        \leq 2\frac{\delta}{11}+\left(\frac{\delta}{11}\right)^2 + \frac{24}{11}\frac{6\|S_2\|_2\|S_1\|_2}{55\sqrt{1-\epsilon_L}}~\delta +\left(  \frac{6\|S_2\|_2\|S_1\|_2}{55\sqrt{1-\epsilon_L}}~\delta\right)^2, \nonumber
  \end{align}
  with probability at least $1-d$, producing result \eqref{eq:orthErrProb1md}.

  On the other hand, observe that \eqref{eq:SDelVRinvBoundAssump} requires not only the assumption that $S_1, S_2$ simultaneously embed $V$ and $S_1V$ respectively, but also that the sketch matrices embed $\Delta \tilde{V}\hat{R}^{-1}$ and $S_1 \Delta \tilde{V}\hat{R}^{-1}$ respectively. Thus, \eqref{eq:SDelVRinvBoundAssump} and \eqref{eq:intermediateBoundOrth} simultaneously hold with probability at least $(1-d)^2$, and the result of applying both of these results together yields,
   \begin{align}
        \|(S_2S_1\hat{Q})^T(S_2S_1\hat{Q}) - I\|_2 & \nonumber \\
    & \hspace{-20mm}
        \leq 2\frac{\delta}{11}+\left(\frac{\delta}{11}\right)^2 + \frac{24}{11}\|S_2S_1\Delta \tilde{V}\hat{R}^{-1}\|_2 
        +\|S_2S_1\Delta \tilde{V}\hat{R}^{-1}\|_2^2 \nonumber \\
        & \hspace{-20mm}
        \leq 2\frac{\delta}{11}+\left(\frac{\delta}{11}\right)^2 + \frac{24}{11}\frac{6\sqrt{1+\varepsilon_H}}{55\sqrt{1-\epsilon_L}}~\delta +\left(  \frac{6\sqrt{1+\varepsilon_H}}{55\sqrt{1-\epsilon_L}}~\delta\right)^2, \label{eq:SorthBoundNicer}
  \end{align}
  with probability at least $(1-d)^2$. A useful consequence of Assumptions \ref{assump:theoremAssumptions} is that 
$$1-\varepsilon_L > \frac{9}{625}(1+\varepsilon_H),$$
and thus
\begin{equation}
    \frac{1+\varepsilon_H}{1-\varepsilon_L} < \frac{625}{9} \Rightarrow \sqrt{ \frac{1+\varepsilon_H}{1-\varepsilon_L}} < \frac{25}{3} ~\cdot \label{eq:1plusEpsover1minusEpsbound}
\end{equation}
Applying \eqref{eq:1plusEpsover1minusEpsbound} to \eqref{eq:SorthBoundNicer} (which holds with probability at least $(1-d)^2$) along with the fact that $\delta \leq 1$ from \eqref{eq:delta} implies $\delta^2 \leq \delta$, and therefore,
\begin{align}
        \|(S_2S_1\hat{Q})^T(S_2S_1\hat{Q}) - I\|_2 & \nonumber \\
        & \hspace{-20mm}
        \leq 2\frac{\delta}{11}+\left(\frac{\delta}{11}\right)^2 + \frac{24}{11}\frac{6\sqrt{1+\varepsilon_H}}{55\sqrt{1-\epsilon_L}}~\delta +\left(  \frac{6\sqrt{1+\varepsilon_H}}{55\sqrt{1-\epsilon_L}}~\delta\right)^2 \nonumber \\
         & \hspace{-20mm}
        \leq \frac{2}{11}\delta+\left(\frac{1}{11}\right)^2 \delta + \frac{24}{11}\frac{6\cdot 25}{55\cdot 3}~\delta +\left(  \frac{6\cdot 25}{55\cdot 3}\right)^2 \delta \nonumber \\
        & \hspace{-20mm}
        = \frac{2\cdot 11 \cdot 55^2 \cdot 3^2 + 55^2 \cdot 3^2 + 24 \cdot 6 \cdot 25 \cdot 11 \cdot 55 \cdot 3 + 6^2\cdot 25^2 \cdot 11^2}{11^2 \cdot 55^2 \cdot 3^2} \delta \nonumber \\
        & \hspace{-20mm}
        = 3\delta~, \nonumber
  \end{align}
   with probability at least $(1-d)^2$, and thus result \eqref{eq:orthErrProb1mdsq} follows.  
\end{proof}



Similar to the analysis of the condition number of $Q$ generated by \texttt{randQR} in exact arithmetic in Section \ref{multi:sec}, we show next that provided that $V$ has full numerical rank, then $\hat{Q}$ generated by \texttt{randQR} in floating point arithmetic also satisfies $\kappa(\hat{Q})~=~O(1)$. 

\begin{theorem}[Conditioning of \texttt{randQR}]
Suppose Assumptions \ref{assump:theoremAssumptions} are satisfied.
Then with probability at least $1-d$, the $\hat{Q}$ matrix obtained with Algorithm~\ref{alg:randQR} 
(\texttt{randQR}) has condition number $\kappa(\hat{Q}) = O(1)$. In fact,
\begin{equation}
    \kappa(\hat{Q}) \leq \frac{33}{25\sqrt{\frac{1-\epsilon_L}{1+\epsilon_H}}-3} \cdot
    \label{eq:condQgeneral}
\end{equation}
\label{thm:randQRcond}
\end{theorem}

\begin{proof}
As a direct consequence of \eqref{eq:VRinvSingularValBound}, \eqref{eq:singValBoundQ}, \eqref{eq:delVRinvfinalbound}, and the fact that $\delta \leq 1$, 
\begin{align*}
    \sigma_m(\hat{Q}) &\geq \sigma_m(V\hat{R}^{-1}) - \|\hat{Q}-V\hat{R}^{-1}\|_2 
    \geq \frac{1-\frac{\delta}{11}}{\sqrt{1+\epsilon_H}} - \frac{6\delta}{55\sqrt{1-\epsilon_L}} \\
    &\geq \frac{10}{11\sqrt{1+\epsilon_H}} - \frac{6}{55\sqrt{1-\epsilon_L}} ~\cdot
\end{align*}
\noindent
Additionally, we found in \eqref{eq:QnrmRoughBound} that 
$$\sigma_1(\hat{Q}) = \| \hat{Q}\|_2 \leq \frac{6}{5\sqrt{1-\epsilon_L}} ~\cdot $$
Thus,
\begin{equation*}
    \kappa(\hat{Q}) = \frac{\sigma_1(\hat{Q})}{\sigma_m(\hat{Q})} \leq \frac{33}{25\sqrt{\frac{1-\epsilon_L}{1+\epsilon_H}}-3}~, 
\end{equation*}
which is the desired result.
Since the intermediate results \eqref{eq:VRinvSingularValBound}, \eqref{eq:singValBoundQ}, \eqref{eq:QnrmRoughBound}, and \eqref{eq:delVRinvfinalbound} simultaneously hold with probability at least $1-d$, as discussed in Section~\ref{sec:probRes}, the final result \eqref{eq:condQgeneral} holds with this probability as well.
\end{proof}

In the following result we show that \texttt{rand\_cholQR}$(V)$ (Algorithm \ref{ajh:alg:sketchQRChol}) produces a factor $\hat{Q}$ that is orthogonal in the Euclidean inner product up to a factor of $O(\textbf{u})$ and has a factorization error of $O(\textbf{u})\|V\|_2$ for any numerically full rank~$V$.

\begin{theorem}[\texttt{rand\_cholQR} Errors]
Suppose Assumptions \ref{assump:theoremAssumptions} are satisfied.
Then with probability at least $1-d$, the $\hat{Q}, \hat{R}$ factors obtained with Algorithm \ref{ajh:alg:sketchQRChol}  (\texttt{rand\_cholQR}) has $O(\textbf{u})$ orthogonality error and $O(\textbf{u})\|V\|_2$ factorization error. More specifically, 
\begin{equation}
    \|\hat{Q}^T\hat{Q}-I\|_2 \leq \frac{5445}{\left( 25\sqrt{\frac{1-\epsilon_L}{1+\epsilon_H}}-3\right)^2}\left(nm+m(m+1)\right)\textbf{u},
    \label{eq:lossOfOrthErrrandCholQRgeneral}
\end{equation}
\vspace*{-4mm}
\begin{align}
    \|V - \hat{Q}\hat{R}\|_2 &\leq \left(\frac{56}{25\frac{1-\varepsilon_L}{\sqrt{1+\epsilon_H}}-3\sqrt{1-\varepsilon_L}} + \frac{1.5}{\sqrt{1-\varepsilon_L}}\sqrt{ 1+\frac{5445(nm+m(m+1))\textbf{u}}{\left( 25\sqrt{\frac{1-\varepsilon_L}{1+\varepsilon_H}}-3\right)^2} } \right)\nonumber \\
    & \quad \left( \sqrt{1+\varepsilon_H}\|V\|_2 + \frac{\sqrt{1-\varepsilon_L}}{12}\sigma_m(V)\delta\right)m^2\textbf{u} + \frac{\delta}{10}\sigma_m(V),
    \label{eq:repErrrandCholQRgeneral}
\end{align}
where $\delta$ is bounded as in \eqref{eq:delta}.
\label{thm:randCholQRorth}
\end{theorem}

\begin{proof}
    In Algorithm \ref{ajh:alg:sketchQRChol}, we obtain $\hat{Q}_0, \hat{R}_0$ from \texttt{randQR} (so that the results in Section \ref{appendix1} apply to $\hat{Q}_0, \hat{R}_0$), and then obtain $\hat{Q}, \hat{R}$ where $\hat{R} = \text{fl}(\hat{R}_1\hat{R_0})$ and $\hat{Q}, \hat{R}_1$ are the outputs of Cholesky QR applied to $\hat{Q}_0$. As a direct consequence of Theorem \ref{thm:randQRcond}, $\hat{Q}_0$ arising from Step 1 of Algorithm \ref{ajh:alg:sketchQRChol} satisfies 
$$\kappa(\hat{Q}_0) \leq \frac{33}{25\sqrt{\frac{1-\epsilon_L}{1+\epsilon_H}}-3} ~\cdot$$
By \cite[Lemma 3.1]{CholeskyQR2ErrAnalysis}, it follows that step 2 of Algorithm \ref{ajh:alg:sketchQRChol} gives $\hat{Q}$ satisfying
    \begin{align*}
        \|\hat{Q}^T\hat{Q} - I \|_2 &\leq \frac{5}{64} 64 \kappa(\hat{Q}_0)^2\left(nm+m(m+1)\right)\textbf{u}\\
        &\leq \frac{5445}{\left( 25\sqrt{\frac{1-\epsilon_L}{1+\epsilon_H}}-3\right)^2}\left(nm+m(m+1)\right)\textbf{u},
    \end{align*}
\noindent
and thus \eqref{eq:lossOfOrthErrrandCholQRgeneral} follows.
    
    Now, notice that by \eqref{eq:cnmuBound}--\eqref{eq:gammaBounds}, 
    \begin{equation}
        \sqrt{\frac{1+\gamma_nm}{1-\gamma_{m+1}m}} \leq \sqrt{\frac{1+1.1nm\textbf{u}}{1-1.1(m+1)m\textbf{u}}} \leq \sqrt{\frac{1.1}{0.9}} \leq 1.11. \label{eq:intermediateRepBound}
    \end{equation}
    
     \noindent Observe that $\hat{R} = \hat{R}_1\hat{R}_0 + \Delta \hat{R}$ where $\hat{R}_1$ is the Cholesky factor of $\hat{Q}_0^T\hat{Q}_0$, where $\hat{Q}_0$ results from \texttt{randQR}, and $|\Delta \hat{R}| < \gamma_m |\hat{R}_1||\hat{R}_0|$ \cite[Eq. (3.13)]{HighamNumAlg}. Then, it follows by \cite[Eq.~(3.16)]{CholeskyQR2ErrAnalysis}, \eqref{eq:gammaBounds}, \eqref{eq:Rbound2Norm}, \eqref{eq:E1plusE2bound}, \eqref{eq:QnrmRoughBound}, and \eqref{eq:intermediateRepBound}, that
    \begin{align}
    \|\Delta \hat{R} \|_2 &\leq \|\Delta \hat{R}\|_F \leq  \gamma_m \|\hat{R}_1\|_F\|\hat{R}_0\|_F \leq m \gamma_m \|\hat{R}_1\|_2 \|\hat{R}_0\|_2 \nonumber \\
    &\leq m\gamma_m \sqrt{\frac{1+\gamma_nm}{1-\gamma_{m+1}m}} \|\hat{Q}_0\|_2 \|\hat{R}_0\|_2 \nonumber \\
    &\leq 1.23m^{2}\textbf{u} \frac{6}{5\sqrt{1-\varepsilon_L}} (\sqrt{1+\varepsilon_H}\|V\|_2+ \frac{\sqrt{1-\varepsilon_L}}{12}\sigma_m(V)\delta ) \nonumber \\
    &\leq m^{2}\textbf{u} \frac{1.5}{\sqrt{1-\varepsilon_L}}(\sqrt{1+\varepsilon_H}\|V\|_2+ \frac{\sqrt{1-\varepsilon_L}}{12}\sigma_m(V)\delta ) ~\cdot \label{eq:delRgemmBound}
    \end{align}
     
     Next, observe that by \eqref{eq:hatQexp} we have that $\hat{Q}_0\hat{R}_0 = V + \Delta \tilde{V}$ from \texttt{randQR}. Using this and \cite[Eq. (3.24)]{CholeskyQR2ErrAnalysis} to bound $\|\hat{Q}_0 - \hat{Q}\hat{R}_1\|_2$,
    \begin{align}
        -\|\Delta \tilde{V}\|_2 + \|V-\hat{Q}\hat{R}\|_2 &\leq \|V+\Delta \tilde{V}-\hat{Q}\hat{R}\|_2 = \| \hat{Q}_0\hat{R}_0 - \hat{Q}\hat{R}_1\hat{R}_0 - \hat{Q}\Delta \hat{R}\|_2 \nonumber \\
        &\leq \|\hat{R}_0\| \|\hat{Q}_0 - \hat{Q}\hat{R}_1\|_2 + \| \hat{Q} \|_2\|\Delta \hat{R}\|_2  \nonumber \\
        &\leq 1.4\|\hat{R}_0\|_2  \kappa(\hat{Q_0})\|\hat{Q}_0\|_2m^2\textbf{u} + \|\hat{Q}\|_2 \|\Delta \hat{R}\|_2  \cdot\label{eq:VmQRrandCholQR}
    \end{align}
\noindent
Note that \eqref{eq:lossOfOrthErrrandCholQRgeneral} implies 
\begin{equation*}
\| \hat{Q} \|_2 \leq \sqrt{1+\frac{5445}{\left( 25\sqrt{\frac{1-\epsilon_L}{1+\epsilon_H}}-3\right)^2}\left(nm+m(m+1)\right)\textbf{u}}~. 
\end{equation*}
\noindent
Additionally, by \eqref{eq:repErr2} in Theorem \ref{thm:randQRorthpt2}, 
\begin{align}
    \|\Delta \tilde{V}\|_2 = \|V - \hat{Q}_0\hat{R}_0 \|_2 \leq \frac{\delta}{10}\sigma_m(V) . \label{eq:DelTildeVBound}
\end{align}
\noindent
Now, starting from \eqref{eq:Rbound2Norm}, we can use \eqref{eq:E1plusE2bound} to obtain
\begin{align}
    \|\hat{R}_0\|_2 &\leq \sqrt{1+\varepsilon_H}\|V\|_2 + \|E_1\|_2 + \|E_2\|_2 \nonumber \\
    &\leq \sqrt{1+\varepsilon_H}\|V\|_2 + \frac{\sqrt{1-\varepsilon_L}}{12}\sigma_m(V)\delta. \label{eq:randCholQRR0bound}
\end{align}

\noindent Using \eqref{eq:condQgeneral} and \eqref{eq:QnrmRoughBound} to bound $\kappa(\hat{Q}_0)$ and $\|\hat{Q}_0\|_2$, \eqref{eq:delRgemmBound} to bound $\|\Delta \hat{R}\|_2$, \eqref{eq:randCholQRR0bound} to bound $\|\hat{R}_0\|_2$, adding $\|\Delta \tilde{V}\|_2$ to both sides of \eqref{eq:VmQRrandCholQR} and then bounding $\| \Delta \tilde{V} \|_2$ using \eqref{eq:DelTildeVBound}, we finally obtain
\begin{align*}
    \|V-\hat{Q}\hat{R}\|_2 &\leq \left( \frac{56}{25\frac{1-\varepsilon_L}{\sqrt{1+\epsilon_H}}-3\sqrt{1-\varepsilon_L}} + \frac{1.5}{\sqrt{1-\varepsilon_L}}\sqrt{ 1+\frac{5445(nm+m(m+1))\textbf{u}}{\left( 25\sqrt{\frac{1-\varepsilon_L}{1+\varepsilon_H}}-3\right)^2} } \right) \\
    & \quad \left( \sqrt{1+\varepsilon_H}\|V\|_2 + \frac{\sqrt{1-\varepsilon_L}}{12}\sigma_m(V)\delta\right)m^2\textbf{u} + \frac{\delta}{10}\sigma_m(V) ,
\end{align*}
which does indeed satisfy $\|V-\hat{Q}\hat{R}\|_2 = O(\textbf{u})\|V\|_2$, since $\sigma_m(V)\delta = O(\textbf{u})\|V\|_2$.

Finally, observe that the probabilistic results used in this proof, namely \eqref{eq:propOneEmbedResult}--\eqref{eq:SDelVRinvBoundNoAssump} and Theorems \ref{thm:randQRorthpt2}--\ref{thm:randQRcond}, simultaneously hold with probability at least $1-d$ (see Section \ref{sec:probRes} for details), and hence \eqref{eq:lossOfOrthErrrandCholQRgeneral} and \eqref{eq:repErrrandCholQRgeneral} hold with this probability as well.
\end{proof}

Theorem \ref{thm:randQRcond} guarantees \texttt{randQR}$(V)$ produces a well-conditioned $\hat{Q}$. We show next that \texttt{rand\_cholQR}$(V)$ produces a factor $\hat{Q}$ with $\kappa(\hat{Q}) \approx 1$ (up to unit roundoff) for any numerically full rank $V$.

\begin{theorem}[Conditioning of \texttt{rand\_cholQR}]
Suppose Assumptions \ref{assump:theoremAssumptions} are satisfied.
Then with probability at least $1-d$, the matrix $\hat{Q}$ obtained with Algorithm \ref{alg:randQR} satisfies $\kappa(\hat{Q}) \approx 1$. More specifically,
\begin{equation}
    \kappa(\hat{Q}) < \sqrt{\frac{1+\frac{5445}{\left( 25\sqrt{\frac{1-\epsilon_L}{1+\epsilon_H}}-3\right)^2}\left(nm+m(m+1)\right)\textbf{u}}{1-\frac{5445}{\left( 25\sqrt{\frac{1-\epsilon_L}{1+\epsilon_H}}-3\right)^2}\left(nm+m(m+1)\right)\textbf{u}}} ~\cdot
    \label{eq:condQrandCholQR}
\end{equation}
Furthermore, if $\frac{5445}{\left( 25\sqrt{\frac{1-\epsilon_L}{1+\epsilon_H}}-3\right)^2}\left(nm+m(m+1)\right)\textbf{u} < \frac{1}{2}$, then 
\begin{equation}
    \kappa(\hat{Q}) < 1+\frac{10890}{\left( 25\sqrt{\frac{1-\epsilon_L}{1+\epsilon_H}}-3\right)^2}\left(nm+m(m+1)\right)\textbf{u}.
    \label{eq:condQrandCholQRsimpler}
\end{equation}
\label{thm:randCholQRcond}
\end{theorem}
\vspace*{-6mm}
\begin{proof}
    It follows from \eqref{eq:lossOfOrthErrrandCholQRgeneral} that the $i^{th}$ eigenvalue of $\hat{Q}^T{Q}$ satisfies
    \begin{align*}
        \lambda_i(\hat{Q}^T\hat{Q}) &\geq 1-\frac{5445}{\left( 25\sqrt{\frac{1-\epsilon_L}{1+\epsilon_H}}-3\right)^2}\left(nm+m(m+1)\right)\textbf{u}, \\
        \lambda_i(\hat{Q}^T\hat{Q}) &\leq 1+\frac{5445}{\left( 25\sqrt{\frac{1-\epsilon_L}{1+\epsilon_H}}-3\right)^2}\left(nm+m(m+1)\right)\textbf{u}.
    \end{align*}
    Thus, the $i^{th}$ singular value of $\hat{Q}$ satisfies
    \begin{align*}
        \sigma_i(\hat{Q}) &\geq \sqrt{1-\frac{5445}{\left( 25\sqrt{\frac{1-\epsilon_L}{1+\epsilon_H}}-3\right)^2}\left(nm+m(m+1)\right)\textbf{u}}, \\
        \sigma_i(\hat{Q}) &\leq \sqrt{1+\frac{5445}{\left( 25\sqrt{\frac{1-\epsilon_L}{1+\epsilon_H}}-3\right)^2}\left(nm+m(m+1)\right)\textbf{u}},
    \end{align*}
    which gives \eqref{eq:condQrandCholQR}. Further, for any $x < \frac{1}{2}$, $\sqrt{\frac{1+x}{1-x}} < 1+2x$, which gives \eqref{eq:condQrandCholQRsimpler}. Since \eqref{eq:lossOfOrthErrrandCholQRgeneral} holds with probability at least $1-d$, \eqref{eq:condQrandCholQR} and \eqref{eq:condQrandCholQRsimpler} hold with this probability as well.
\end{proof}

Theorems \ref{thm:randQRorthpt2}--\ref{thm:randCholQRcond} correspond to multisketchings, that is, to
the application of one sketch matrix after another. In the rest of the section,
we recast our error bounds for a single sketch matrix in Corollaries \ref{cor:randQRorth}--\ref{corr:randCholQRcond}. The results apply for
a single $(\varepsilon, d, m)$ oblivious $\ell_2$-subspace embedding for any $\varepsilon \in [0,\frac{616}{634})$, covering nearly the entire range of possible $\varepsilon \in [0,1)$ for such embeddings. 

We prove all the Corollaries simultaneously, as they are direct consequences of Theorems \ref{thm:randQRorthpt2}--\ref{thm:randCholQRcond} by exploiting the fact that a single sketch can be recast as a product of two sketches, one of which is the identity, which is by definition a $(0,0,m)$ oblivious $\ell_2$-subspace embedding.

\begin{assumption}
Suppose $\varepsilon \in [0,\frac{616}{634})$ and $S \in \mathbb{R}^{p \times m}$ is a $(\varepsilon, d, m)$ oblivious $\ell_2$-subspace embedding. Further, suppose $V \in \mathbb{R}^{n \times m}$ has full rank and \linebreak $1 < m \leq s \leq n$ where $nm\textbf{u} \leq \frac{1}{12}$, $p^{3/2}\textbf{u} \leq \frac{1}{12}$, and
\begin{equation}
   \delta = \frac{383\left( pm^{3/2} + \sqrt{m}(p^{3/2}\sqrt{1+\varepsilon}+n\|S\|_F)\right) }{\sqrt{1-\varepsilon}}\textbf{u}~\kappa(V) \leq 1. \label{eq:delta1sketch}
\end{equation}
\label{assump:corrAssumptions}
\end{assumption} 


\begin{corollary}[\texttt{randQR} Errors]
Suppose Assumptions \ref{assump:corrAssumptions} are satisfied.
Then the $\hat{Q}, \hat{R}$ factors obtained with Algorithm \ref{alg:randQR} (\texttt{randQR}) satisfy
\begin{equation}
    \|V - \hat{Q}\hat{R}\|_2 \leq \frac{\delta}{10}\sigma_m(V),
    \label{eq:repErr21sketch} \nonumber
\end{equation}
and
\begin{align}
        \|(S\hat{Q})^T(S\hat{Q}) - I\|_2 
        &\leq \frac{2\delta}{11}+\left(\frac{\delta}{11}\right)^2 + \frac{24}{11}\frac{6\|S\|_2}{55\sqrt{1-\epsilon}}\delta +\left(  \frac{6\|S\|_2}{55\sqrt{1-\epsilon}}\delta\right)^2 \label{eq:orthErrProb1md1sketch} \nonumber
  \end{align}
  with probability at least $1-d$, where $\delta$ is defined as in \eqref{eq:delta1sketch}. Furthermore, 
  \begin{align}
        \|(S\hat{Q})^T(S\hat{Q}) - I\|_2 &\leq 3\delta  \nonumber 
  \end{align}
  with probability at least $(1-d)^2$.
\label{cor:randQRorth}
\end{corollary}

\begin{corollary}[Conditioning of \texttt{randQR}]
Suppose Assumptions \ref{assump:corrAssumptions} are satisfied.
Then with probability at least $1-d$, the $\hat{Q}$ matrix obtained with Algorithm~\ref{alg:randQR} 
(\texttt{randQR}) has condition number $\kappa(\hat{Q}) = O(1)$. In fact,
\begin{equation}
    \kappa(\hat{Q}) \leq \frac{33}{25\sqrt{\frac{1-\epsilon}{1+\epsilon}}-3} \cdot
    \label{eq:condQgeneral1sketch} \nonumber
\end{equation}
Therefore, if $\varepsilon \leq 0.9$,
\begin{equation*}
    \kappa(\hat{Q}) \leq 12.07.
    \label{eq:condQ1Sketch}
\end{equation*}
\label{cor:randQRcond}
\end{corollary}

\vspace*{-7mm}
\begin{corollary}[\texttt{rand\_cholQR} Errors]
Suppose Assumptions \ref{assump:corrAssumptions} are satisfied.
Then with probability at least $1-d$, the $\hat{Q}, \hat{R}$ factors obtained with Algorithm \ref{ajh:alg:sketchQRChol}  (\texttt{rand\_cholQR}) has $O(\textbf{u})$ orthogonality error and $O(\textbf{u})\|V\|_2$ factorization error. In fact, 
\begin{equation}
    \|\hat{Q}^T\hat{Q}-I\|_2 \leq \frac{5445}{\left( 25\sqrt{\frac{1-\epsilon}{1+\epsilon}}-3\right)^2}\left(nm+m(m+1)\right)\textbf{u},
    \label{eq:lossOfOrthErrrandCholQRgeneral1sketch} \nonumber
\end{equation}
\vspace*{-4mm}
\begin{align}
    \|V - \hat{Q}\hat{R}\|_2 &\leq \left( \frac{56}{25\frac{1-\varepsilon}{\sqrt{1+\epsilon}}-3\sqrt{1-\varepsilon}} + \frac{1.5}{\sqrt{1-\varepsilon}}\sqrt{ 1+\frac{5445(nm+m(m+1))\textbf{u}}{\left( 25\sqrt{\frac{1-\varepsilon}{1+\varepsilon}}-3\right)^2} } \right)\nonumber \\
    & \quad \left( \sqrt{1+\varepsilon}\|V\|_2 + \frac{\sqrt{1-\varepsilon}}{12}\sigma_m(V)\delta\right)m^2\textbf{u} + \frac{\delta}{10}\sigma_m(V),
    \label{eq:repErrrandCholQRgeneral1sketch} \nonumber
\end{align}
where $\delta$ is bounded as in \eqref{eq:delta1sketch}.
\label{corr:randCholQRorth}
\end{corollary}

\begin{corollary}[Conditioning of \texttt{rand\_cholQR}]
Suppose Assumptions \ref{assump:corrAssumptions} are satisfied.
Then with probability at least $1-d$, the matrix $\hat{Q}$ obtained with Algorithm \ref{alg:randQR} satisfies $\kappa(\hat{Q}) \approx 1$. In fact,
\begin{equation}
    \kappa(\hat{Q}) < \sqrt{\frac{1+\frac{5445}{\left( 25\sqrt{\frac{1-\epsilon}{1+\epsilon}}-3\right)^2}\left(nm+m(m+1)\right)\textbf{u}}{1-\frac{5445}{\left( 25\sqrt{\frac{1-\epsilon}{1+\epsilon}}-3\right)^2}\left(nm+m(m+1)\right)\textbf{u}}} \cdot
    \label{eq:condQrandCholQR.cor} \nonumber
\end{equation}
Furthermore, if $\frac{5445}{\left( 25\sqrt{\frac{1-\epsilon}{1+\epsilon}}-3\right)^2}\left(nm+m(m+1)\right)\textbf{u} < \frac{1}{2}$, then 
\begin{equation}
    \kappa(\hat{Q}) < 1+\frac{10890}{\left( 25\sqrt{\frac{1-\epsilon}{1+\epsilon}}-3\right)^2}\left(nm+m(m+1)\right)\textbf{u}.
    \label{eq:condQrandCholQRsimpler.cor} \nonumber
\end{equation}
\label{corr:randCholQRcond}
\end{corollary}

\begin{proof}
    We prove Corollaries \ref{cor:randQRorth}--\ref{corr:randCholQRcond} simultaneously by
considering one subspace embedding $S = S_1$ is equivalent to two subspace embeddings $S_2S_1$ simply by interpreting $S_2 = I_{p,p}$ as the $p \times p$ identity, which is by definition a $(0,0,m)$ oblivious $\ell_2$-subspace embedding, therefore giving $\varepsilon_1 = \varepsilon_H = \varepsilon_L = \varepsilon$, $d = d_1$, $p = p_2 = p_1$, $\varepsilon_2 = 0$ and $d_2 = 0$.

We show next that if  $\varepsilon = \varepsilon_L = \varepsilon_H \in [0, \frac{616}{634})$,
then $\varepsilon_L \in [0, \frac{616}{625}-\frac{9}{625}\varepsilon_H)$.
Indeed, $\varepsilon_L \in [0, \frac{616}{625}-\frac{9}{625}\varepsilon_H)$ is equivalent in this case to $0 \leq \varepsilon < \frac{616}{625}-\frac{9}{625}\varepsilon$,\linebreak or $0 \leq \frac{634}{625} \varepsilon < \frac{616}{625}$, or what is the same, $\varepsilon \in [0, \frac{616}{634})$.
This means that when $\varepsilon_H = \varepsilon_L = \varepsilon \in [0, \frac{616}{634})$,
    Assumptions \ref{assump:theoremAssumptions} imply Assumptions \ref{assump:corrAssumptions}. Thus, Assumptions~\ref{assump:theoremAssumptions} are satisfied
    and Corollaries \ref{cor:randQRorth}--\ref{corr:randCholQRcond} are direct consequences of Theorems~\ref{thm:randQRorthpt2}--~\ref{thm:randCholQRcond}.
\end{proof}

\begin{remark}
    { \rm
    Observe that \eqref{eq:delta1sketch} in Assumptions \ref{assump:corrAssumptions} is identical to \eqref{eq:delta} in Assumptions \ref{assump:theoremAssumptions} using $S_1 = S$ and $S_2 = I_{p, p}$, where $p_1 = p_2 = p$. However, the analysis in Section \ref{sec:matMulErr} of $\|E_1\|_2$ takes into account roundoff errors for two matrix multiplications for two sketches to compute $\hat{W}$, while in the case of Corollaries \ref{cor:randQRorth}--\ref{corr:randCholQRcond}, only one sketch and therefore one matrix multiplication to compute $\hat{W}$ is necessary. Therefore, we are over-estimating the error $\|E_1\|_2$ in the single sketch case, and if the analysis in Sections \ref{appendix1} were carefully performed again, we could tighten the bound on $\delta$ in \eqref{eq:delta1sketch}, thereby loosening the requirements on $\kappa(V)$ in Assumptions \ref{assump:corrAssumptions}. However, asymptotically, the requirement on $\kappa(V)$ would ultimately still be that $\delta~\leq~g(n,m,p_1,p_2)\textbf{u}~\kappa(V) \leq 1$ for some low-degree polynomial $g$.
    }
\end{remark}

\section{Numerical Experiments} \label{numer:sec}

We conducted numerical experiments with two goals in mind. First, we compare the performance of \texttt{rand\_cholQR} with the performance of \texttt{cholQR2}, \texttt{sCholQR3}, and Householder QR on a high-performance GPU leveraging vendor-optimized libraries. Second, we empirically validate the bounds given in Section \ref{sec:stateKeyResults}, and more generally, compare the stability of \texttt{rand\_cholQR} to the stability of \texttt{cholQR2}, \texttt{sCholQR3}, and Householder QR.

\subsection{Implementation Details}

We implemented \texttt{rand\_cholQR}, \texttt{cholQR2}, \texttt{sCholQR3}, and Householder QR in C\texttt{++}. To be portable to a GPU, we used the Kokkos Performance Portability Library \cite{Kokkos} and Kokkos Kernels~\cite{kokkos-kernels}. For our experiments on an NVIDIA GPU, we configured and built our code such that Kokkos Kernels calls NVIDIA's cuBLAS and cuSPARSE linear algebra libraries for optimized dense and sparse basic linear algebra routines \cite{cuBLAS, cuSPARSE}. To perform LAPACK routines that are not currently available natively within Kokkos Kernels (i.e., \texttt{dgeqrf} and \texttt{dorgqr} for computing the Householder QR factorization, and \texttt{dpotrf} for the Cholesky factorization), we directly called NVIDIA's cuSOLVER linear algebra library \cite{lapackUserGuide, cuda, cuSOLVER}. Test results were obtained using Kokkos 3.7.01, Cuda 11.7.99, and GCC 7.2.0 on an AMD EPYC 7742 64-Core 2.25GHz CPU with an NVIDIA A100-SXM4 40GB GPU. All computations were done in double precision, so $\textbf{u} = 2^{-52} \approx 10^{-16}$.

\subsection{Performance Results}\label{sec:perfResults}

We tested \texttt{rand\_cholQR} with a variety of sketching strategies. The simplest was the case of a Gaussian sketch $S = \frac{1}{\sqrt{p}}G \in \mathbb{R}^{p \times n}$, 
which was chosen to embed $V \in \mathbb{R}^{n \times m}$ with distortion $\epsilon = 0.99$, using a sketch size of 
$p = \lceil 36.01\log(m) \rceil$
to produce a $(0.99,1/m, m)$ oblivious $\ell_2$-subspace embedding \cite[Lemma 4.1]{GaussSize}. We tested the CountSketch by explicit construction of the sparse matrix and applied it using a sparse-matrix vector product. The sketch size used with a $S \in \mathbb{R}^{p \times n}$ CountSketch matrix was $p = \lceil 6.8 (m^2+m) \rceil$, which can be shown to be a $(0.99, 0.15, m)$ oblivious $\ell_2$-subspace embedding \cite[Theorem 1]{CountSketchSize}.

In our implementation of multisketching, we chose $S_1 \in \mathbb{R}^{p_1\times n}$ as a CountSketch with $\varepsilon_1 = 0.9$ requiring sketch size $p_1 = \lceil 8.24 (m^2+m) \rceil$ to produce a $(0.9,0.15,m)$ oblivious $\ell_2$-subspace embedding, and $S_2 \in \mathbb{R}^{p_2 \times p_1}$ a Gaussian sketch with $p_2 = \lceil 74.3\log(p_1) \rceil$ giving $\varepsilon_2 = 0.49$ to give a $(0.49, 1/m, m)$ oblivious $\ell_2$-subspace embedding. Overall, $S_2S_1$ produced an embedding with $\varepsilon_L \approx 0.9490$, $\varepsilon_H \approx 1.8310$, and $d \approx 0.15$. It is easily verified that $S_2S_1$ is in line with Assumptions~\ref{assump:theoremAssumptions}, and that both of $S_1$ and $S_2$ satisfy 
Assumptions~\ref{assump:corrAssumptions}, ensuring the analysis in Section~\ref{sec:stateKeyResults} is relevant to the experiments. Runtimes of \texttt{rand\_cholQR} did not include the time to generate the sketch, as this was assumed to be a fixed overhead time.


\begin{figure}
\begin{subfigure}{0.5\textwidth}
  \centering
  \includegraphics[width=\linewidth]{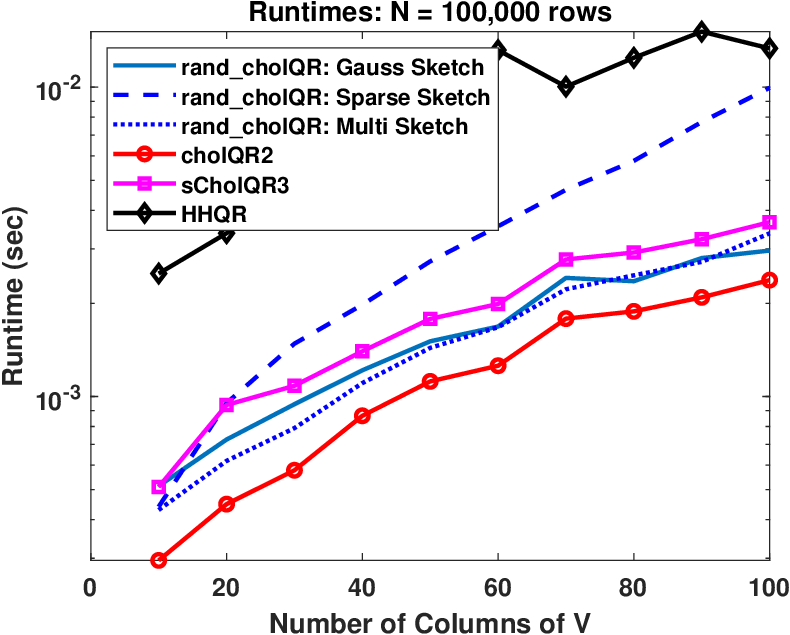}
  \caption{Raw runtimes}
  \label{fig:perf_1e5}
\end{subfigure}
\begin{subfigure}{0.5\textwidth}
  \centering
  \includegraphics[width=\linewidth]{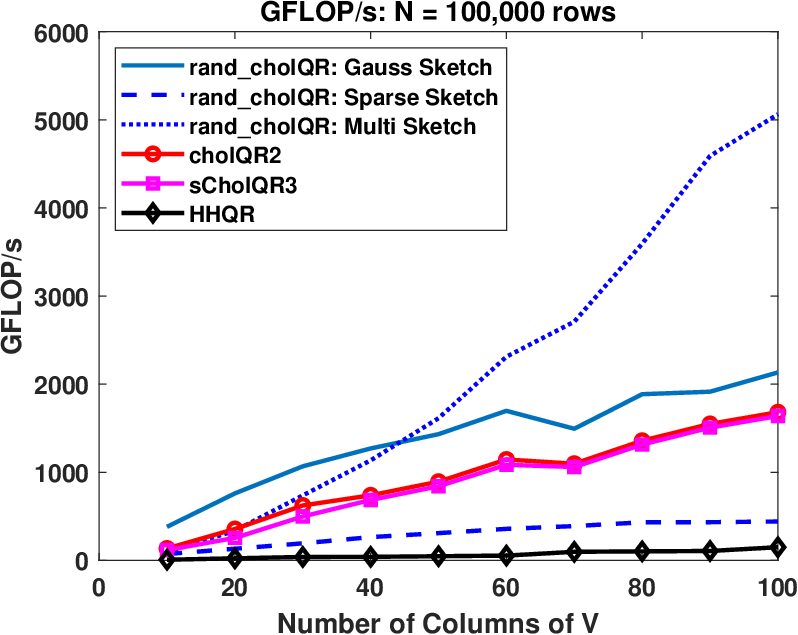}
  \caption{GFLOP/s}
  \label{fig:gflops_1e5}
\end{subfigure}
\begin{center}
\begin{subfigure}{0.5\textwidth}
  \centering
  \includegraphics[width=\linewidth]{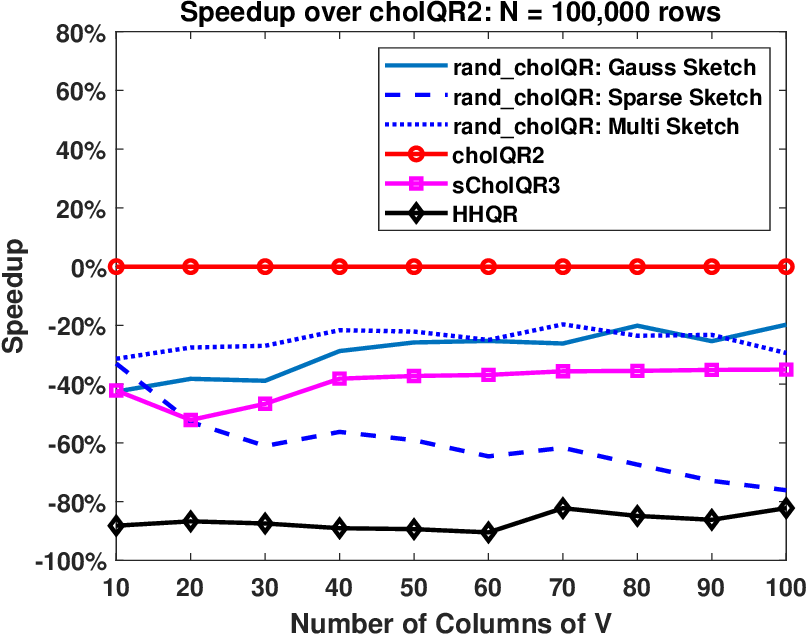}
  \caption{Speedup over \texttt{cholQR2}}
  \label{fig:speedup_1e5}
\end{subfigure}
\end{center}%
\caption[Runtimes of QR factorizations of $V$ with a fixed number of rows as the number of columns vary]{Performance of QR algorithms on matrices with $n = 100,000$ rows.}
\label{fig:perf_results_1e5}
\end{figure}

\begin{figure}
\begin{subfigure}{0.5\textwidth}
  \centering
  \includegraphics[width=\linewidth]{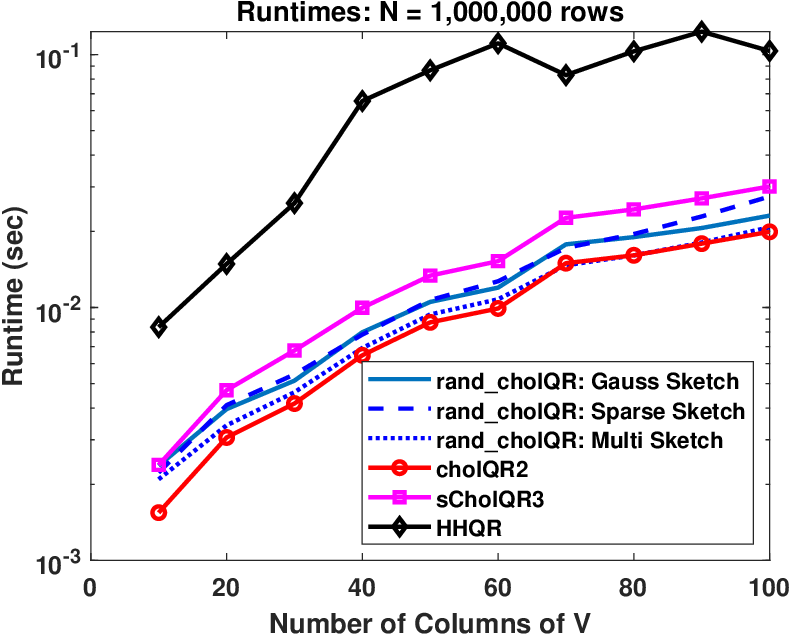}
  \caption{Raw runtimes}
  \label{fig:perf_1e6}
\end{subfigure}
\begin{subfigure}{0.5\textwidth}
  \centering
  \includegraphics[width=\linewidth]{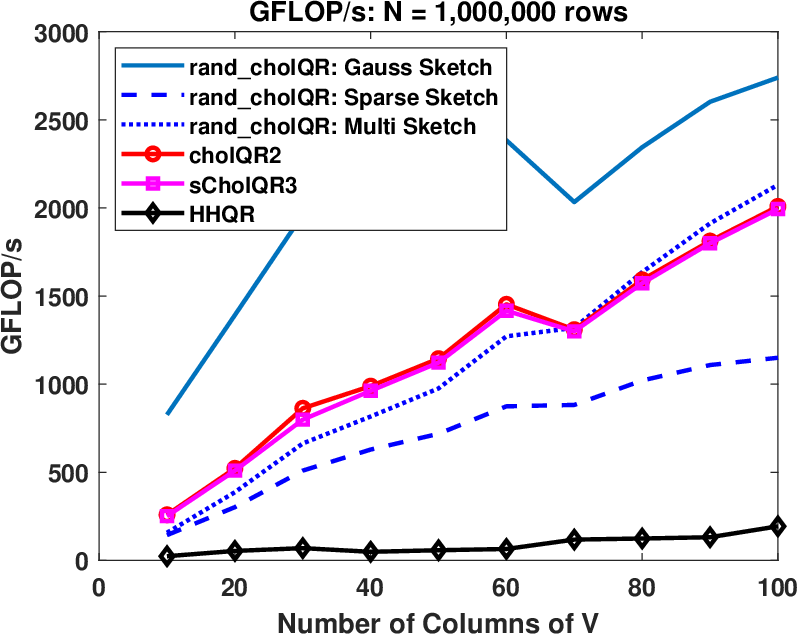}
  \caption{GFLOP/s}
  \label{fig:gflops_1e6}
\end{subfigure}
\begin{center}
\begin{subfigure}{0.5\textwidth}
  \centering
  \includegraphics[width=\linewidth]{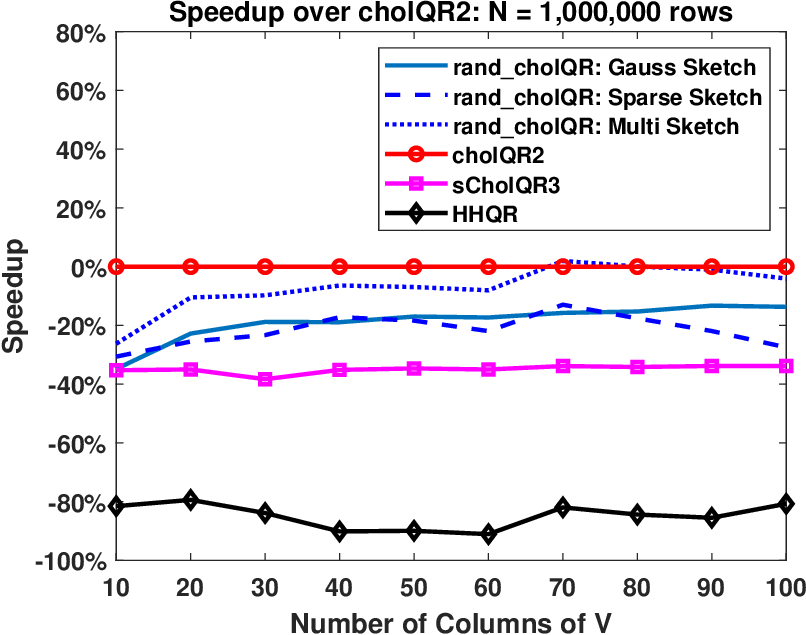}
  \caption{Speedup over \texttt{cholQR2}}
  \label{fig:speedup_1e6}
\end{subfigure}
\end{center}%
\caption[Runtimes of QR factorizations of $V$ with a fixed number of rows as the number of columns vary]{Performance of QR algorithms on matrices with $n = 1,000,000$ rows.}
\label{fig:perf_results_1e6}
\end{figure}

\begin{figure}
\begin{subfigure}{0.5\textwidth}
  \centering
  \includegraphics[width=\linewidth]{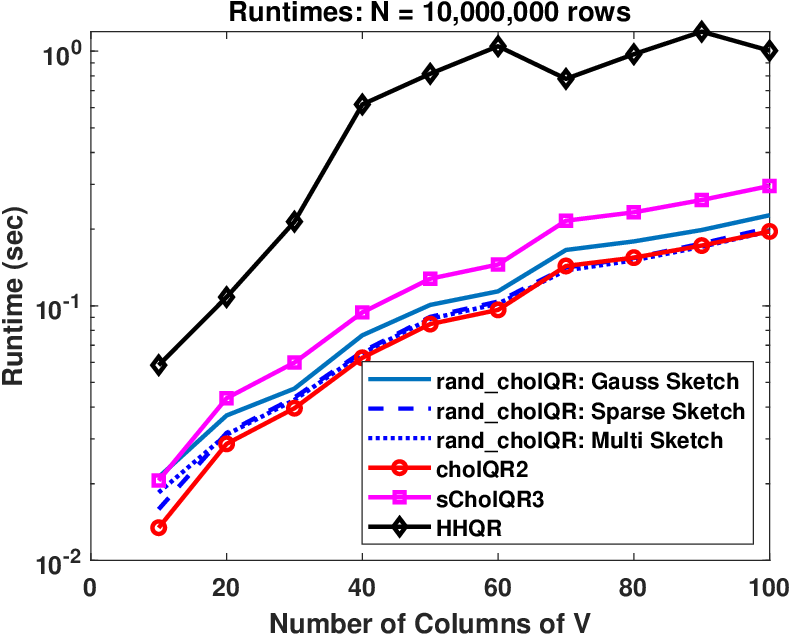}
  \caption{Raw runtimes}
  \label{fig:perf_1e7}
\end{subfigure}
\begin{subfigure}{0.5\textwidth}
  \centering
  \includegraphics[width=\linewidth]{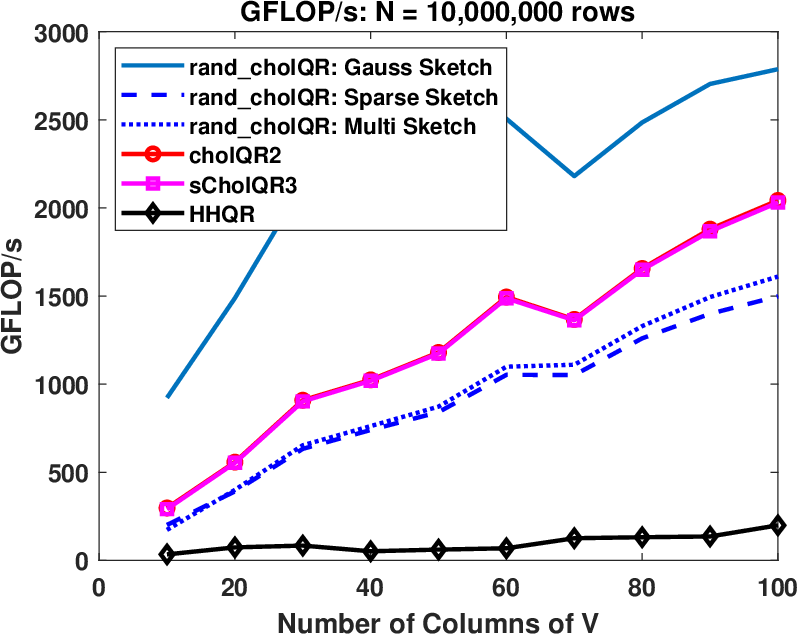}
  \caption{GFLOP/s}
  \label{fig:gflops_1e7}
\end{subfigure}
\begin{center}
\begin{subfigure}{0.5\textwidth}
  \centering
  \includegraphics[width=\linewidth]{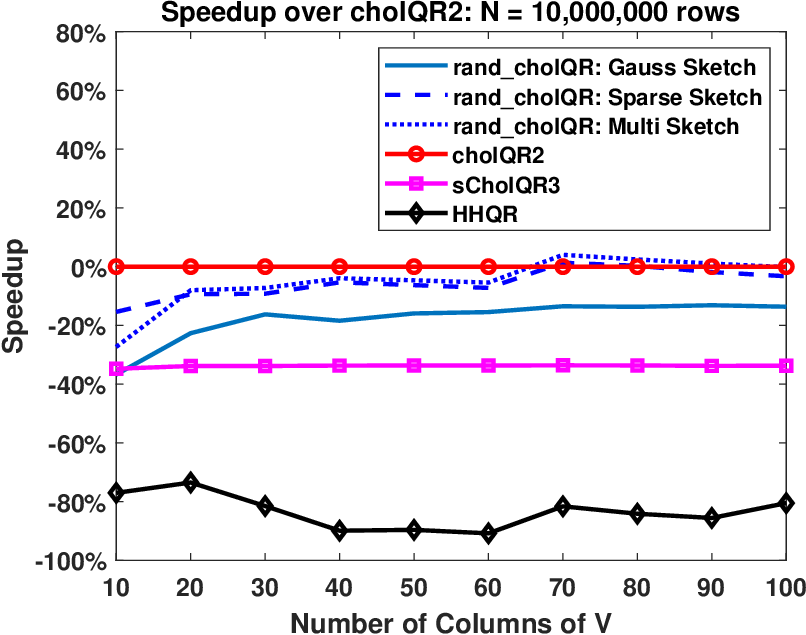}
  \caption{Speedup over \texttt{cholQR2}}
  \label{fig:speedup_1e7}
\end{subfigure}
\end{center}%
\caption[Runtimes of QR factorizations of $V$ with a fixed number of rows as the number of columns vary]{Performance of QR algorithms on matrices with $n = 10,000,000$ rows.}
\label{fig:perf_results_1e7}
\end{figure}


Figures \ref{fig:perf_results_1e5}--\ref{fig:perf_results_1e7} show the performance of each QR method for test problems with $n = 10^5-10^7$ rows and $m = 10$--$100$ columns. Within each figure, subfigure (a) shows the runtimes of each algorithm, and subfigure (b) shows the GFLOP/s of each algorithm.  Since \texttt{cholQR2} is typically expected to be the fastest algorithm, subfigure (c) shows the relative speedup of each QR method compared to \texttt{cholQR2}. 
 
 Figures \ref{fig:perf_results_1e5}--\ref{fig:perf_results_1e7} indicate that \texttt{cholQR2} is indeed the fastest method in general, while multisketch \texttt{rand\_cholQR} performs the closest to \texttt{cholQR2}. Additionally, Figures~\ref{fig:perf_results_1e6} (c) and \ref{fig:perf_results_1e7} (c) show that in some cases, multisketch \texttt{rand\_cholQR} actually outperforms \texttt{cholQR2}, specifically for $n = 10^6$ rows and $m = 70$ columns, and for $n 
= 10^7$ rows and $m = 70$--$90$ columns. The most notable result is that for $n = 10^7$ rows and $m = 70$ columns, multisketch \texttt{rand\_cholQR} is 4\% faster than \texttt{cholQR2}. Multisketch \texttt{rand\_cholQR} is significantly faster than \texttt{sCholQR3}, as evidenced by Figures~\ref{fig:perf_results_1e5}--\ref{fig:perf_results_1e7}, and both algorithms have the same $O(\textbf{u})\kappa(V) < 1$ stability requirement.

Subfigure (b) of Figures \ref{fig:perf_results_1e5}--\ref{fig:perf_results_1e7} demonstrate that the implementations and algorithms are indeed high-performance, with multisketched \texttt{rand\_cholQR} achieving up to 5,000 GFLOP/s on a single GPU. The GFLOP/s of multisketched \texttt{rand\_cholQR} is so high in Figure \ref{fig:eps_results_1e5} because $p_1$, the number of columns used in the Gaussian matrix, is nearly the same as $n$. Thus, it does more computations than both the single-sketched Gaussian version of the algorithm and \texttt{cholQR2}, in spite of the fact that it does not take much more time than either of these algorithms. Note that in these plots, \texttt{cholQR2} and \texttt{sCholQR3} are nearly overlapping, as they perform the exact same operations, but \texttt{sCholQR3} does 50\% more computation.

\subsection{Empirical Error Analysis}\label{sec:numResults}
Figure \ref{fig:cond_test} shows the orthogonalization error $\| I - \hat{Q}^T\hat{Q} \|_F$ and the relative factorization error $\| V - \hat{Q}\hat{R}\|_F/\|V\|_F$ for condition number $\kappa(V) \in [1,10^{16}]$. The results demonstrate that \texttt{rand\_cholQR} maintains $O(\textbf{u})$ orthogonality error and $O(\textbf{u})\|V\|_2$ factorization error\footnote{This follows because $\| V - \hat{Q}\hat{R}\|_F/\|V\|_F = O(\textbf{u})$.} while $\kappa(V) < O(\textbf{u}^{-1})$, as predicted by Theorem~\ref{thm:randCholQRorth}, and is more robust than \texttt{cholQR2} and \texttt{sCholQR3}. In practice, it appears that \texttt{rand\_cholQR} is stable even when $V$ is numerically rank-deficient. In summary, Figures \ref{fig:perf_results_1e5}--\ref{fig:cond_test} demonstrate that multisketch \texttt{rand\_cholQR} significantly improves the robustness of \texttt{cholQR2} and \texttt{sCholQR3} at little to no cost, therefore making \texttt{rand\_cholQR} a superior high-performance QR algorithm.
Additionally, observe that, as indicated by a large dot, lines for \texttt{cholQR2} and \texttt{sCholQR3} end at $\kappa(V) = 10^8$ and $\kappa(V) = 10^{12}$ respectively, as the methods fail beyond these points.

\begin{figure}
  \begin{subfigure}{0.5\textwidth}
  \centering
  \includegraphics[width=\linewidth]{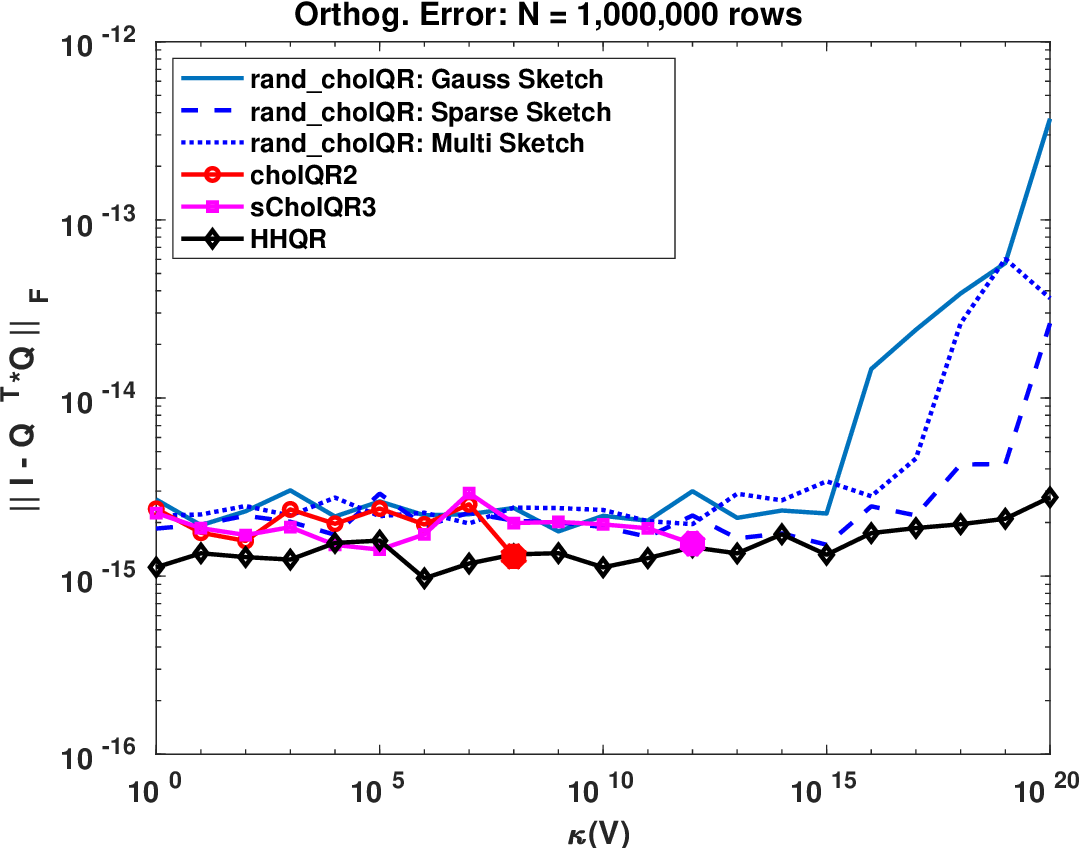}
  \caption{Orthogonality Error}
  \label{fig:orth_err}
\end{subfigure}
\begin{subfigure}{0.5\textwidth}
  \centering
  \includegraphics[width=\linewidth]{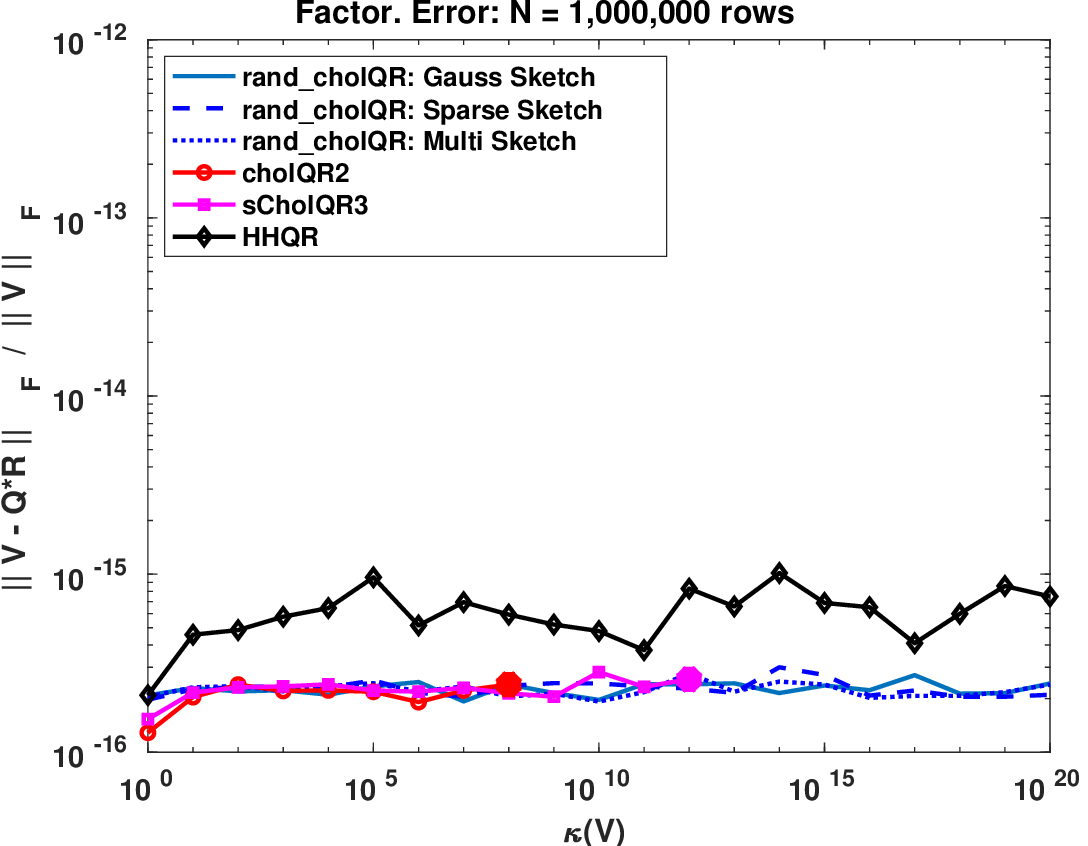}
  \caption{Relative Factorization Error}
  \label{fig:fact_err}
\end{subfigure}%
\caption[Orthogonality and relative factorization error of QR factorizations of a matrix $V$ with varying condition number.]{Orthogonality (left) and relative factorization error (right) of the QR factorization of a matrix $V$ with varying condition number. To explicitly control $\kappa(V)$, $V := L\Sigma R^T \in \mathbb{R}^{n \times m}$ using random orthogonal matrices $L, R$, and a diagonal $\Sigma$ with log-equispaced entries in the range $[\kappa^{-\frac{1}{2}}(V),\kappa^{\frac{1}{2}}(V)]$. Indicated by a large dot, lines for \texttt{cholQR2} and \texttt{sCholQR3} end at $\kappa(V) = 10^8$ and $\kappa(V) = 10^{12}$ respectively, as the methods fail beyond these points.}
\label{fig:cond_test}
\end{figure}

\subsection{Effect of Larger Distortion Factor $\epsilon$ in \texttt{rand\_cholQR}}\label{sec:epsResults}

In Section \ref{sec:relWork}, we claimed that the stability guarantees for $\epsilon \approx 1$ is an improvement over existing results only providing stability guarantees up to $\epsilon \approx 0.5$ as the larger values of $\epsilon$ allow for better performance. In this subsection, we include experiments that justify this and highlight why theoretical results for larger values of $\epsilon$ should be studied further in the context of randomized QR factorizations.

In Figures \ref{fig:eps_results_1e5}--\ref{fig:eps_results_1e7}, we begin by demonstrating the speedup attained by using a single sketch with $\epsilon = 0.99$ over a sketch with $\epsilon = 0.49$ where subfigures (a) and (b) correspond to the Gaussian sketch and CountSketch experiments, respectively. Thereafter, we demonstrate the speedup in attained in the multisketch case using a larger $\epsilon_1 = 0.9$ compared to $\epsilon_1 = 0.49$ in subfigure (c). In the multisketch case, we keep $\epsilon_2 = 0.49$ fixed, and did not push $\epsilon_1$ over 0.9 to ensure Assumptions \ref{assump:theoremAssumptions} and therefore the theoretical results in Section \ref{sec:stateKeyResults} hold.

Specifically, in the case of the Gaussian experiments (Figures \ref{fig:eps_results_1e5}--\ref{fig:eps_results_1e7} (a)), we used $p = \lceil 36.01\log(m) \rceil$ to attain distortion factor $\epsilon = 0.99$ and $p = \lceil 74.3\log(m) \rceil$ to attain distortion factor $\epsilon = 0.49$\footnote{These sketch sizes can be shown to produce $(0.99,1/m, m)$ and $(0.49,1/m, m)$ oblivious $\ell_2$-subspace embeddings, respectively}. In the case of the CountSketch (Figures \ref{fig:eps_results_1e5}--\ref{fig:eps_results_1e7} (b)), we used sketch size $p = \lceil 6.8 (m^2+m) \rceil$ to attain $\epsilon = 0.99$ and sketch size $p = \lceil 27.8 (m^2+m) \rceil$ to attain $\epsilon = 0.49$ \footnote{This can be shown to attain $(0.99, 0.15, m)$ and $(0.49, 0.15, m)$ oblivious $\ell_2$-subspace embeddings, respectively.}. Finally, in the multisketch case (Figures \ref{fig:eps_results_1e5}--\ref{fig:eps_results_1e7} (c)), the first sketch was a CountSketch with embedding dimension $p_1 = \lceil 8.24 (m^2+m) \rceil$ to attain $\varepsilon_1 = 0.9$ and embedding dimension $p_1 = \lceil 27.8 (m^2+m) \rceil$ to attain $\epsilon = 0.49$\footnote{These choices can be shown to produce $(0.9,0.15,m)$ and $(0.49,0.15,m)$ oblivious $\ell_2$-subspace embeddings, respectively}.


\begin{figure}
\begin{subfigure}{0.5\textwidth}
  \centering
  \includegraphics[width=\linewidth]{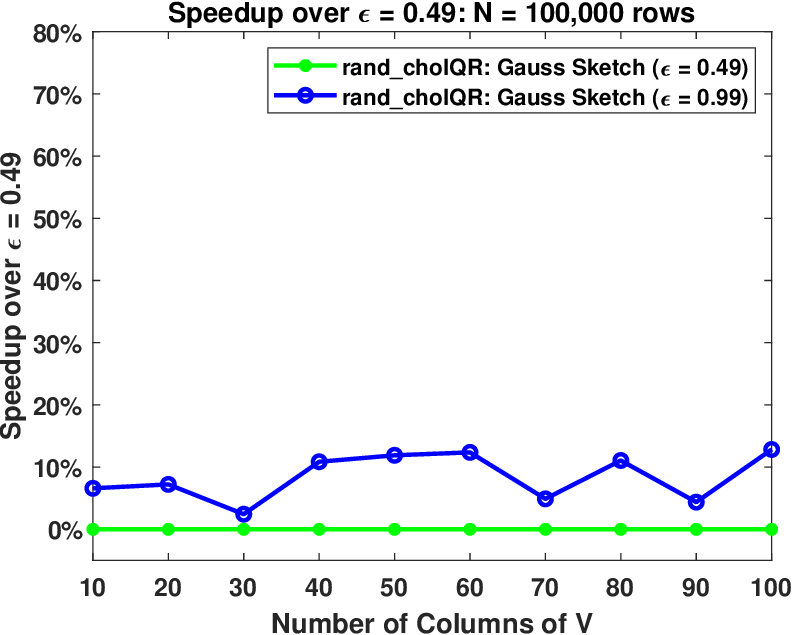}
  \caption{Gaussian Sketch}
  \label{fig:gauss_speedup_1e5}
\end{subfigure}
\begin{subfigure}{0.5\textwidth}
  \centering
  \includegraphics[width=\linewidth]{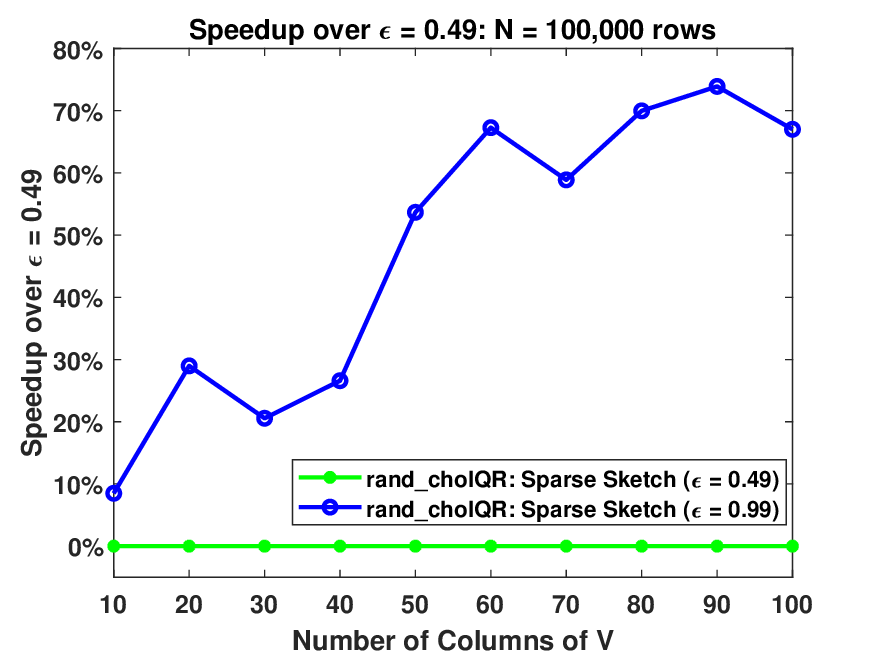}
  \caption{Sparse Sketch}
  \label{fig:sparse_speedup_1e5}
\end{subfigure}
\begin{center}
\begin{subfigure}{0.5\textwidth}
  \centering
  \includegraphics[width=\linewidth]{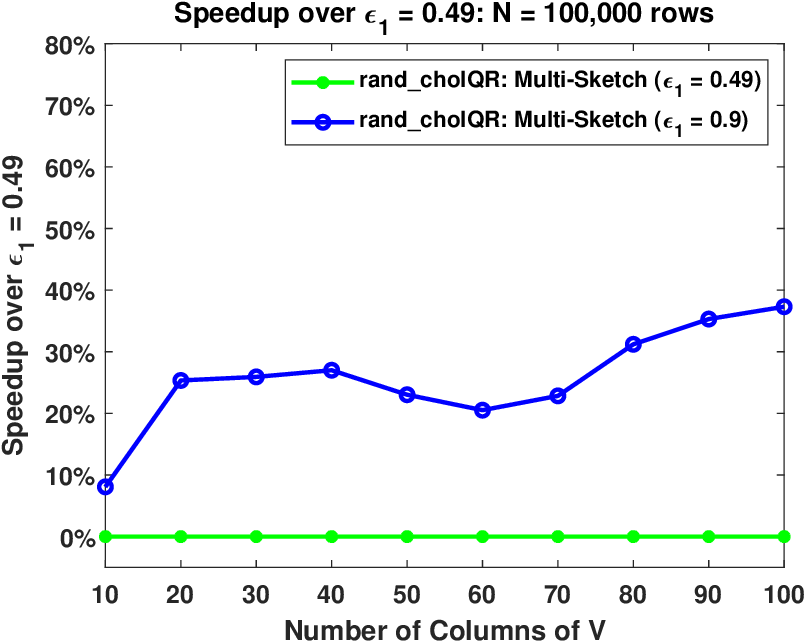}
  \caption{Multi-Sketch}
  \label{fig:multi_speedup_1e5}
\end{subfigure}
\end{center}%
\caption[Runtimes of QR factorizations of $V$ with a fixed number of rows as the number of columns vary]{Speedup attained in \texttt{rand\_cholQR} using different values of $\epsilon$, $n = 100,000$.}
\label{fig:eps_results_1e5}
\end{figure}

\begin{figure}
\begin{subfigure}{0.5\textwidth}
  \centering
  \includegraphics[width=\linewidth]{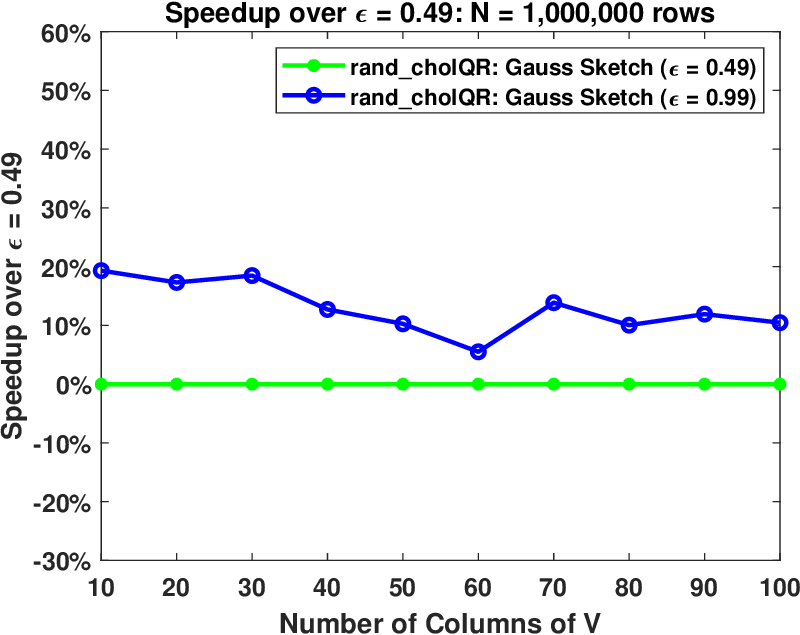}
  \caption{Gaussian Sketch}
  \label{fig:gauss_speedup_1e6}
\end{subfigure}
\begin{subfigure}{0.5\textwidth}
  \centering
  \includegraphics[width=\linewidth]{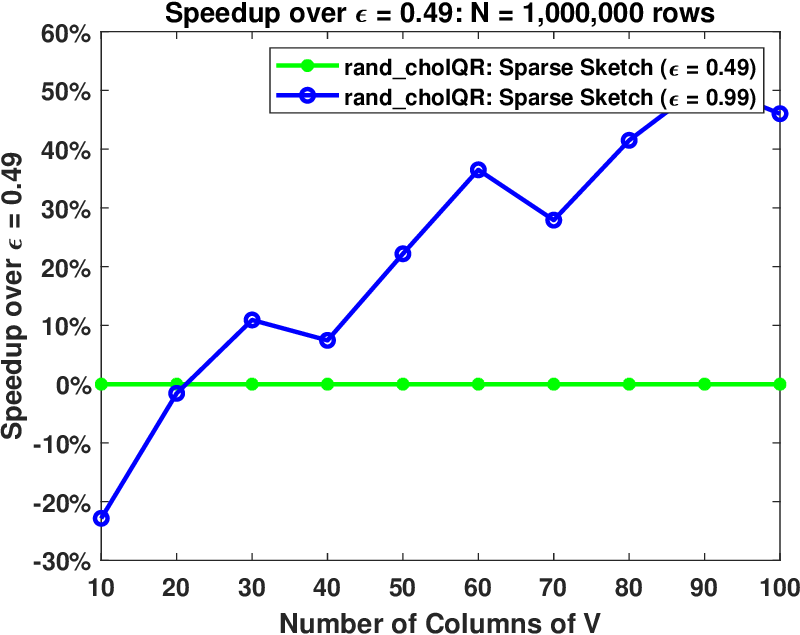}
  \caption{Sparse Sketch}
  \label{fig:sparse_speedup_1e6}
\end{subfigure}
\begin{center}
\begin{subfigure}{0.5\textwidth}
  \centering
  \includegraphics[width=\linewidth]{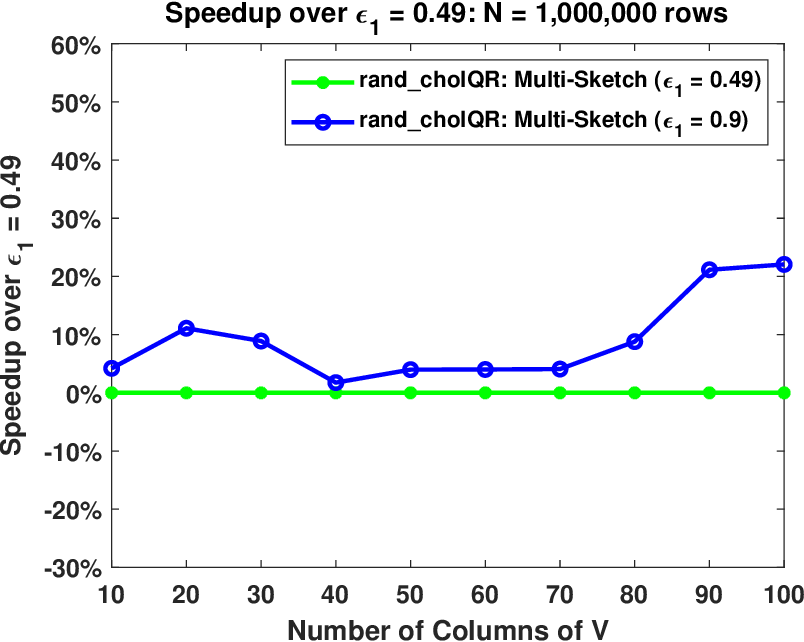}
  \caption{Multi-Sketch}
  \label{fig:multi_speedup_1e6}
\end{subfigure}
\end{center}%
\caption[Runtimes of QR factorizations of $V$ with a fixed number of rows as the number of columns vary]{Speedup attained in \texttt{rand\_cholQR} using different values of $\epsilon$, $n = 1,000,000$.}
\label{fig:eps_results_1e6}
\end{figure}

\begin{figure}
\begin{subfigure}{0.5\textwidth}
  \centering
  \includegraphics[width=\linewidth]{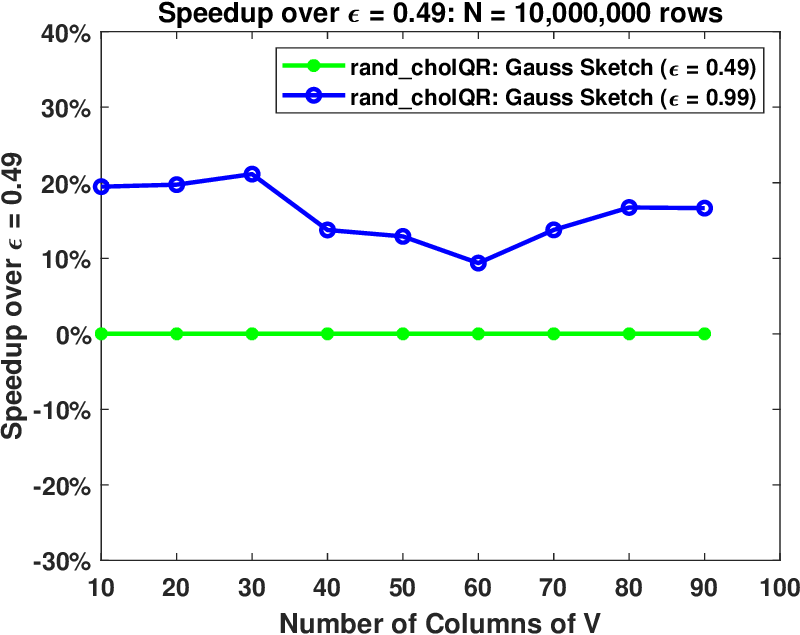}
  \caption{Gaussian Sketch}
  \label{fig:gauss_speedup_1e7}
\end{subfigure}
\begin{subfigure}{0.5\textwidth}
  \centering
  \includegraphics[width=\linewidth]{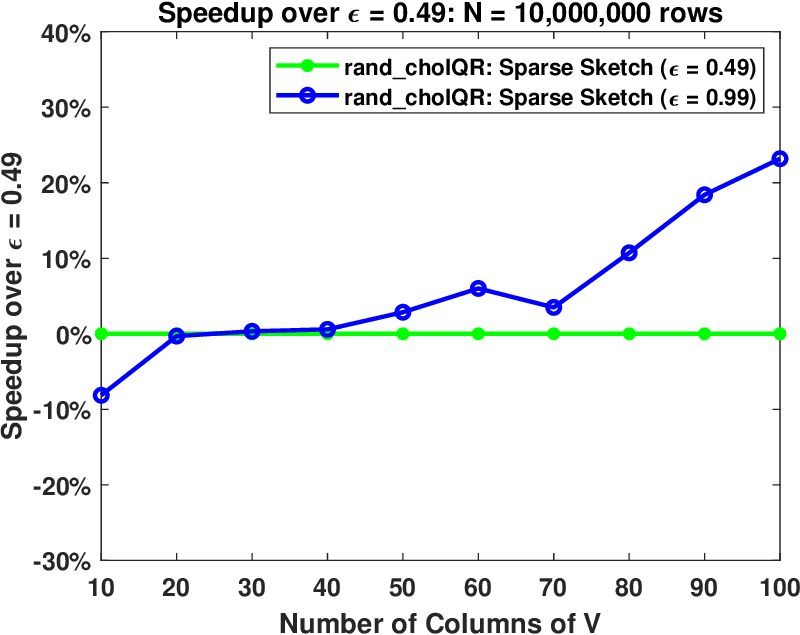}
  \caption{Sparse Sketch}
  \label{fig:sparse_speedup_1e7}
\end{subfigure}
\begin{center}
\begin{subfigure}{0.5\textwidth}
  \centering
  \includegraphics[width=\linewidth]{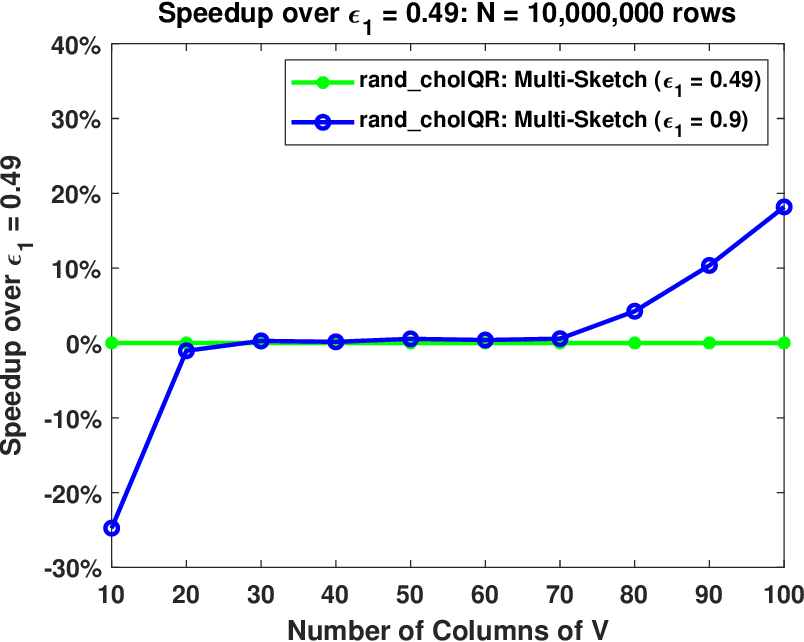}
  \caption{Multi-Sketch}
  \label{fig:multi_speedup_1e7}
\end{subfigure}
\end{center}%
\caption[Runtimes of QR factorizations of $V$ with a fixed number of rows as the number of columns vary]{Speedup attained in \texttt{rand\_cholQR} using different values of $\epsilon$, $n = 10,000,000$.}
\label{fig:eps_results_1e7}
\end{figure}

As predicted, the results indicate that the in most cases, the impact of choosing a larger $\epsilon$ allows for significant runtime improvements in the randomized QR routine tested. Presumably, these results would translate to other randomized methods where a large distortion is acceptable, such as in pre-processing/preconditioning procedures or inexact orthogonalization schemes, where only linear independence is strictly necessary.

\section{Conclusions and Future Work}

The results in Section~\ref{sec:stateKeyResults} indicate that \texttt{rand\_cholQR} using one or two sketch matrices orthogonalizes any numerically full-rank matrix $V$ up to $O(\textbf{u})$ error.
\linebreak
This is a significant improvement over CholeskyQR2, which requires 
\linebreak
$\kappa(V) \lessapprox \textbf{u}^{-1/2}$ to ensure a stable factorization. Our results for a single sketch apply for any $\varepsilon$-embedding with $\varepsilon \in [0, \frac{616}{634})$, covering nearly the entire possible range for $\varepsilon$-embeddings.

Our performance results in Section \ref{numer:sec} indicate that the significantly better stability properties of \texttt{rand\_cholQR} over \texttt{cholQR2} come at virtually no increase in the factorization time on a modern GPU. Additionally, \texttt{rand\_cholQR} is theoretically just as stable and in practice more stable than \texttt{sCholQR3}, while being substantially faster. This is due to the fact that \texttt{rand\_cholQR} and \texttt{cholQR2} incur the same number of processor synchronizations, while leveraging mostly BLAS-3 or optimized sparse matrix-vector routines for most of the required computation. In fact, \texttt{rand\_cholQR} can perform better than \texttt{cholQR2} when using the multisketch framework. Of the sketching strategies considered, the multisketch framework is the most advantageous, likely because it requires little additional storage compared to \texttt{cholQR2}, and applying the sketches in this framework is extremely cheap.

Future work includes applying \texttt{rand\_cholQR} to Krylov subspace methods that require tall-and-skinny QR factorizations, particularly block~\cite{Gutknecht, OLearyBCG}, 
\linebreak
\mbox{$s$-step}~\cite{sStepCG, Hoemmen, sStepGMRES}, and enlarged Krylov methods~\cite{EnlargedKrylov}, 
and further investigations into efficient multisketching implementations on a GPU, as our analysis is amenable to any multisketching strategy (not just a CountSketch followed by a dense Gaussian). In particular, applying the CountSketch matrix could potentially be optimized better than using a sparse-matrix vector multiplication by using a custom routine to add/subtract subsets of randomly selected rows in parallel using batched BLAS-1 routines, which should be investigated. Additionally, the performance of \texttt{randQR} and \texttt{rand\_cholQR} using dense Rademacher sketch matrices in place of dense Gaussian sketches as in \cite{GaussSize} should be investigated, as Rademacher sketches impose far lower storage requirements than a Gaussian sketch and can be generated much more efficiently.

\section*{Acknowledgements}

The authors appreciate very much the comments and questions by the handling editor Ilse Ipsen, and by the three
anonymous referees, which helped improved our presentation.

Sandia National Laboratories is a multimission laboratory managed and operated by National Technology and Engineering Solutions of Sandia, LLC., a wholly owned subsidiary of Honeywell International, Inc., for the U.S. Department of Energy’s National Nuclear Security Administration under contract DE-NA-0003525. This work was in part supported by the Exascale Computing Project (17-SC-20-SC), a collaborative effort of the U.S. Department of Energy Office of Science and the National Nuclear Security Administration.

\section*{Statements and Declarations}

\paragraph{Competing interests.}

The authors declare no competing interests.
 
\bibliography{ref.bib}
\newpage 
\appendix

\section{Raw Runtimes From Experiments}

\begin{table}[h]
\centering
\resizebox{\columnwidth}{!}{%
\begin{tabular}{c|rrrrrrrrrr}
 & \multicolumn{10}{c}{Columns}\\ Algorithm & 10 & 20 & 30 & 40 & 50 & 60 & 70 & 80 & 90 & 100 \\\hline\texttt{rand\_cholQR}: Gauss& \multirow{2}{*}{0.0005}& \multirow{2}{*}{0.0007}& \multirow{2}{*}{0.0009}& \multirow{2}{*}{0.0012}& \multirow{2}{*}{0.0015}& \multirow{2}{*}{0.0017}& \multirow{2}{*}{0.0024}& \multirow{2}{*}{0.0024}& \multirow{2}{*}{0.0028}& \multirow{2}{*}{0.0030}\\ 
($\epsilon = 0.99$) & & & & & & & & & & \\ 
\texttt{rand\_cholQR}: Gauss& \multirow{2}{*}{0.0005}& \multirow{2}{*}{0.0008}& \multirow{2}{*}{0.0010}& \multirow{2}{*}{0.0014}& \multirow{2}{*}{0.0017}& \multirow{2}{*}{0.0019}& \multirow{2}{*}{0.0025}& \multirow{2}{*}{0.0027}& \multirow{2}{*}{0.0029}& \multirow{2}{*}{0.0034}\\ 
($\epsilon = 0.49$) & & & & & & & & & & \\ 
\texttt{rand\_cholQR}: Sparse& \multirow{2}{*}{0.0004}& \multirow{2}{*}{0.0010}& \multirow{2}{*}{0.0015}& \multirow{2}{*}{0.0020}& \multirow{2}{*}{0.0027}& \multirow{2}{*}{0.0036}& \multirow{2}{*}{0.0047}& \multirow{2}{*}{0.0058}& \multirow{2}{*}{0.0077}& \multirow{2}{*}{0.0100}\\ 
($\epsilon = 0.99$) & & & & & & & & & & \\ 
\texttt{rand\_cholQR}: Sparse& \multirow{2}{*}{0.0005}& \multirow{2}{*}{0.0013}& \multirow{2}{*}{0.0019}& \multirow{2}{*}{0.0027}& \multirow{2}{*}{0.0059}& \multirow{2}{*}{0.0108}& \multirow{2}{*}{0.0113}& \multirow{2}{*}{0.0192}& \multirow{2}{*}{0.0296}& \multirow{2}{*}{0.0302}\\ 
($\epsilon = 0.49$) & & & & & & & & & & \\ 
\texttt{rand\_cholQR}: Multi& \multirow{2}{*}{0.0004}& \multirow{2}{*}{0.0006}& \multirow{2}{*}{0.0008}& \multirow{2}{*}{0.0011}& \multirow{2}{*}{0.0014}& \multirow{2}{*}{0.0017}& \multirow{2}{*}{0.0022}& \multirow{2}{*}{0.0025}& \multirow{2}{*}{0.0027}& \multirow{2}{*}{0.0034}\\ 
($\epsilon_1 = 0.9, \epsilon_2 = 0.49$) & & & & & & & & & & \\ 
\texttt{rand\_cholQR}: Multi& \multirow{2}{*}{0.0005}& \multirow{2}{*}{0.0008}& \multirow{2}{*}{0.0011}& \multirow{2}{*}{0.0015}& \multirow{2}{*}{0.0019}& \multirow{2}{*}{0.0021}& \multirow{2}{*}{0.0029}& \multirow{2}{*}{0.0036}& \multirow{2}{*}{0.0042}& \multirow{2}{*}{0.0054}\\ 
($\epsilon_1 = 0.49, \epsilon_2 = 0.49$) & & & & & & & & & & \\ 
\multirow{2}{*}{\texttt{cholQR2}}& \multirow{2}{*}{0.0003}& \multirow{2}{*}{0.0004}& \multirow{2}{*}{0.0006}& \multirow{2}{*}{0.0009}& \multirow{2}{*}{0.0011}& \multirow{2}{*}{0.0013}& \multirow{2}{*}{0.0018}& \multirow{2}{*}{0.0019}& \multirow{2}{*}{0.0021}& \multirow{2}{*}{0.0024}\\ 
\\ 
\multirow{2}{*}{\texttt{sCholQR3}}& \multirow{2}{*}{0.0005}& \multirow{2}{*}{0.0009}& \multirow{2}{*}{0.0011}& \multirow{2}{*}{0.0014}& \multirow{2}{*}{0.0018}& \multirow{2}{*}{0.0020}& \multirow{2}{*}{0.0028}& \multirow{2}{*}{0.0029}& \multirow{2}{*}{0.0032}& \multirow{2}{*}{0.0037}\\ 
\\ 
\multirow{2}{*}{Householder}& \multirow{2}{*}{0.0005}& \multirow{2}{*}{0.0009}& \multirow{2}{*}{0.0011}& \multirow{2}{*}{0.0014}& \multirow{2}{*}{0.0018}& \multirow{2}{*}{0.0020}& \multirow{2}{*}{0.0028}& \multirow{2}{*}{0.0029}& \multirow{2}{*}{0.0032}& \multirow{2}{*}{0.0037}\\ 
\\ 

\end{tabular}%
}
\caption{Raw runtimes (in seconds) of all experiments with $n = 100,000$ rows.
} \label{tab:runtimes_1e5}
\end{table}

\begin{table}[h]
\centering
\resizebox{\columnwidth}{!}{%
\begin{tabular}{c|rrrrrrrrrr}
 & \multicolumn{10}{c}{Columns}\\ Algorithm & 10 & 20 & 30 & 40 & 50 & 60 & 70 & 80 & 90 & 100 \\ \hline
 \texttt{rand\_cholQR}: Gauss& \multirow{2}{*}{0.0024}& \multirow{2}{*}{0.0040}& \multirow{2}{*}{0.0051}& \multirow{2}{*}{0.0080}& \multirow{2}{*}{0.0105}& \multirow{2}{*}{0.0120}& \multirow{2}{*}{0.0178}& \multirow{2}{*}{0.0190}& \multirow{2}{*}{0.0206}& \multirow{2}{*}{0.0231}\\ 
($\epsilon = 0.99$) & & & & & & & & & & \\ 
\texttt{rand\_cholQR}: Gauss& \multirow{2}{*}{0.0029}& \multirow{2}{*}{0.0048}& \multirow{2}{*}{0.0063}& \multirow{2}{*}{0.0091}& \multirow{2}{*}{0.0117}& \multirow{2}{*}{0.0127}& \multirow{2}{*}{0.0206}& \multirow{2}{*}{0.0211}& \multirow{2}{*}{0.0234}& \multirow{2}{*}{0.0258}\\ 
($\epsilon = 0.49$) & & & & & & & & & & \\ 
\texttt{rand\_cholQR}: Sparse& \multirow{2}{*}{0.0022}& \multirow{2}{*}{0.0041}& \multirow{2}{*}{0.0054}& \multirow{2}{*}{0.0078}& \multirow{2}{*}{0.0107}& \multirow{2}{*}{0.0127}& \multirow{2}{*}{0.0172}& \multirow{2}{*}{0.0195}& \multirow{2}{*}{0.0229}& \multirow{2}{*}{0.0275}\\ 
($\epsilon = 0.99$) & & & & & & & & & & \\ 
\texttt{rand\_cholQR}: Sparse& \multirow{2}{*}{0.0018}& \multirow{2}{*}{0.0040}& \multirow{2}{*}{0.0061}& \multirow{2}{*}{0.0084}& \multirow{2}{*}{0.0138}& \multirow{2}{*}{0.0200}& \multirow{2}{*}{0.0239}& \multirow{2}{*}{0.0334}& \multirow{2}{*}{0.0469}& \multirow{2}{*}{0.0509}\\ 
($\epsilon = 0.49$) & & & & & & & & & & \\ 
\texttt{rand\_cholQR}: Multi& \multirow{2}{*}{0.0021}& \multirow{2}{*}{0.0034}& \multirow{2}{*}{0.0046}& \multirow{2}{*}{0.0069}& \multirow{2}{*}{0.0094}& \multirow{2}{*}{0.0108}& \multirow{2}{*}{0.0147}& \multirow{2}{*}{0.0161}& \multirow{2}{*}{0.0181}& \multirow{2}{*}{0.0208}\\ 
($\epsilon_1 = 0.9, \epsilon_2 = 0.49$) & & & & & & & & & & \\ 
\texttt{rand\_cholQR}: Multi& \multirow{2}{*}{0.0022}& \multirow{2}{*}{0.0038}& \multirow{2}{*}{0.0051}& \multirow{2}{*}{0.0070}& \multirow{2}{*}{0.0098}& \multirow{2}{*}{0.0112}& \multirow{2}{*}{0.0153}& \multirow{2}{*}{0.0176}& \multirow{2}{*}{0.0229}& \multirow{2}{*}{0.0266}\\ 
($\epsilon_1 = 0.49, \epsilon_2 = 0.49$) & & & & & & & & & & \\ 
\multirow{2}{*}{\texttt{cholQR2}}& \multirow{2}{*}{0.0015}& \multirow{2}{*}{0.0031}& \multirow{2}{*}{0.0042}& \multirow{2}{*}{0.0065}& \multirow{2}{*}{0.0087}& \multirow{2}{*}{0.0099}& \multirow{2}{*}{0.0150}& \multirow{2}{*}{0.0161}& \multirow{2}{*}{0.0179}& \multirow{2}{*}{0.0199}\\ 
\\ 
\multirow{2}{*}{\texttt{sCholQR3}}& \multirow{2}{*}{0.0024}& \multirow{2}{*}{0.0047}& \multirow{2}{*}{0.0068}& \multirow{2}{*}{0.0100}& \multirow{2}{*}{0.0134}& \multirow{2}{*}{0.0152}& \multirow{2}{*}{0.0226}& \multirow{2}{*}{0.0244}& \multirow{2}{*}{0.0270}& \multirow{2}{*}{0.0301}\\ 
\\ 
\multirow{2}{*}{Householder}& \multirow{2}{*}{0.0024}& \multirow{2}{*}{0.0047}& \multirow{2}{*}{0.0068}& \multirow{2}{*}{0.0100}& \multirow{2}{*}{0.0134}& \multirow{2}{*}{0.0152}& \multirow{2}{*}{0.0226}& \multirow{2}{*}{0.0244}& \multirow{2}{*}{0.0270}& \multirow{2}{*}{0.0301}\\ 
\\ 
\end{tabular}%
}
\caption{Raw runtimes (in seconds) of all experiments with $n = 1,000,000$ rows. 
} \label{tab:runtimes_1e6}
\end{table}

\begin{table}[h]
\centering
\resizebox{\columnwidth}{!}{%
\begin{tabular}{c|rrrrrrrrrr}
 & \multicolumn{10}{c}{Columns}\\ Algorithm & 10 & 20 & 30 & 40 & 50 & 60 & 70 & 80 & 90 & 100 \\
 \hline
 \texttt{rand\_cholQR}: Gauss& \multirow{2}{*}{0.0212}& \multirow{2}{*}{0.0371}& \multirow{2}{*}{0.0473}& \multirow{2}{*}{0.0766}& \multirow{2}{*}{0.1008}& \multirow{2}{*}{0.1140}& \multirow{2}{*}{0.1657}& \multirow{2}{*}{0.1790}& \multirow{2}{*}{0.1984}& \multirow{2}{*}{0.2268}\\ 
($\epsilon = 0.99$) & & & & & & & & & & \\ 
\texttt{rand\_cholQR}: Gauss& \multirow{2}{*}{0.0264}& \multirow{2}{*}{0.0462}& \multirow{2}{*}{0.0600}& \multirow{2}{*}{0.0888}& \multirow{2}{*}{0.1157}& \multirow{2}{*}{0.1258}& \multirow{2}{*}{0.1921}& \multirow{2}{*}{0.2150}& \multirow{2}{*}{0.2380} & \multirow{2}{*}{}\\ 
($\epsilon = 0.49$) & & & & & & & & & & \\ 
\texttt{rand\_cholQR}: Sparse& \multirow{2}{*}{0.0159}& \multirow{2}{*}{0.0317}& \multirow{2}{*}{0.0436}& \multirow{2}{*}{0.0660}& \multirow{2}{*}{0.0904}& \multirow{2}{*}{0.1039}& \multirow{2}{*}{0.1414}& \multirow{2}{*}{0.1541}& \multirow{2}{*}{0.1756}& \multirow{2}{*}{0.2025}\\ 
($\epsilon = 0.99$) & & & & & & & & & & \\ 
\texttt{rand\_cholQR}: Sparse& \multirow{2}{*}{0.0147}& \multirow{2}{*}{0.0316}& \multirow{2}{*}{0.0438}& \multirow{2}{*}{0.0664}& \multirow{2}{*}{0.0931}& \multirow{2}{*}{0.1105}& \multirow{2}{*}{0.1466}& \multirow{2}{*}{0.1726}& \multirow{2}{*}{0.2152}& \multirow{2}{*}{0.2635}\\ 
($\epsilon = 0.49$) & & & & & & & & & & \\ 
\texttt{rand\_cholQR}: Multi& \multirow{2}{*}{0.0185}& \multirow{2}{*}{0.0312}& \multirow{2}{*}{0.0427}& \multirow{2}{*}{0.0650}& \multirow{2}{*}{0.0889}& \multirow{2}{*}{0.1019}& \multirow{2}{*}{0.1378}& \multirow{2}{*}{0.1508}& \multirow{2}{*}{0.1706}& \multirow{2}{*}{0.1962}\\ 
($\epsilon_1 = 0.9, \epsilon_2 = 0.49$) & & & & & & & & & & \\ 
\texttt{rand\_cholQR}: Multi& \multirow{2}{*}{0.0148}& \multirow{2}{*}{0.0309}& \multirow{2}{*}{0.0428}& \multirow{2}{*}{0.0651}& \multirow{2}{*}{0.0894}& \multirow{2}{*}{0.1023}& \multirow{2}{*}{0.1386}& \multirow{2}{*}{0.1575}& \multirow{2}{*}{0.1903}& \multirow{2}{*}{0.2398}\\ 
($\epsilon_1 = 0.49, \epsilon_2 = 0.49$) & & & & & & & & & & \\ 
\multirow{2}{*}{\texttt{cholQR2}}& \multirow{2}{*}{0.0134}& \multirow{2}{*}{0.0287}& \multirow{2}{*}{0.0396}& \multirow{2}{*}{0.0625}& \multirow{2}{*}{0.0848}& \multirow{2}{*}{0.0964}& \multirow{2}{*}{0.1433}& \multirow{2}{*}{0.1545}& \multirow{2}{*}{0.1724}& \multirow{2}{*}{0.1958}\\ 
\\ 
\multirow{2}{*}{\texttt{sCholQR3}}& \multirow{2}{*}{0.0206}& \multirow{2}{*}{0.0433}& \multirow{2}{*}{0.0598}& \multirow{2}{*}{0.0942}& \multirow{2}{*}{0.1278}& \multirow{2}{*}{0.1453}& \multirow{2}{*}{0.2158}& \multirow{2}{*}{0.2328}& \multirow{2}{*}{0.2603}& \multirow{2}{*}{0.2955}\\ 
\\ 
\multirow{2}{*}{Householder}& \multirow{2}{*}{0.0206}& \multirow{2}{*}{0.0433}& \multirow{2}{*}{0.0598}& \multirow{2}{*}{0.0942}& \multirow{2}{*}{0.1278}& \multirow{2}{*}{0.1453}& \multirow{2}{*}{0.2158}& \multirow{2}{*}{0.2328}& \multirow{2}{*}{0.2603}& \multirow{2}{*}{0.2955}\\ 
\\

\end{tabular}%
}
\caption{Raw runtimes (in seconds) of all experiments with $n = 10,000,000$ rows. Blank entry indicates the experiment could not fit in GPU memory.
} \label{tab:runtimes_1e7}
\end{table}

\end{document}